\documentclass[11pt]{amsart}
\usepackage{amssymb,amsmath,amsthm}
\usepackage{mathrsfs,a4wide}
\usepackage[scrtime]{prelim2e}
\usepackage{epsfig}
\usepackage{appendix}
\theoremstyle{plain}
\newtheorem{theorem}{Theorem}[section]
\newtheorem{lemma}[theorem]{Lemma}
\newtheorem{proposition}[theorem]{Proposition}
\newtheorem{corollary}[theorem]{Corollary}
\newtheorem{remark}[theorem]{Remark}
\newtheorem{definition}{Definition}
\theoremstyle{definition}
\numberwithin{equation}{section}

\newenvironment{oss}{\begin{remark} \begin{rm}}{\end{rm} \end{remark}}

\newcommand{\hs}{{\mathcal H}}

\newcommand{\sinf}{S^{n-1}(\infty)}
\newcommand{\E}{{\mathcal E}}

\newcommand{\F}{{\mathcal F}}

\newcommand{\R}{{\mathbb R}}
\newcommand{\N}{{\mathbb N}}

\newcommand{\Hyp}{{\mathbb H}}

\newcommand{\Om}{\Omega}

\newcommand{\weakstar}{\stackrel{\ast}{\rightharpoonup}}

\newcommand{\weak}{\rightharpoonup}

\newcommand{\e}{\varepsilon}
\newcommand{\epsm}{{\varepsilon_m}}

\newcommand{\f}{\varphi}

\newcommand{\dise}{\tilde d}

\newcommand{\res}{\mathop{\hbox{\vrule height 7pt width .5pt depth 0pt
\vrule height .5pt width 6pt depth 0pt}}\nolimits}

\title
[Phase transitions and minimal hypersurfaces in hyperbolic space] {Phase transitions and minimal hypersurfaces in  hyperbolic space}
\author[A. Pisante]
{Adriano Pisante}
\author[M. Ponsiglione]
{Marcello Ponsiglione}
\address[A. Pisante]{Dipartimento di Matematica ``G. Castelnuovo", Sapienza, Universit\'a di Roma, Piazzale A. Moro 2, 00185 Roma, Italy} 
\email[M. Pisante]{pisante@mat.uniroma1.it}
\address[M. Ponsiglione]{Dipartimento di Matematica ``G. Castelnuovo", Sapienza, Universit\'a di Roma,
Piazzale A. Moro 2, 00185 Roma, Italy} \email[M.
Ponsiglione]{ponsigli@mat.uniroma1.it}

\begin{document}
\vskip .2truecm

\begin{abstract}
\small{The purpose of this paper is to investigate the Cahn-Hillard approximation for entire minimal hypersurfaces in the hyperbolic space. Combining comparison principles with minimization and blow-up arguments, we prove existence results for entire local minimizers with prescribed behaviour at infinity. Then, we study the limit as the length scale tends to zero through a $\Gamma$-convergence analysis, obtaining existence of entire minimal hypersurfaces with prescribed boundary at infinity.
In particular, we recover some existence results proved in \cite{A1} and \cite{La} using geometric measure theory.  
}
\end{abstract}
\maketitle

\vskip -.5truecm
{\small \tableofcontents}
\section{Introduction}
Let  $(\Hyp^n,g)$ be the hyperbolic space with its standard metric $g$, represented either with the Poincar\'e ball or with the half space model. Given a double well potential $W:\R\to\R$, e.g., $W(u)= \frac{1}{4 } (1-u^2)^2$ and $W_\e(u)=\frac1{\e^2}W(u)$, \, $\e>0$, we consider the energy functional 
\begin{equation} \label{enfun}
\E_\e(u,A):= \int_A \frac{1}{2} \|\nabla_g u\|^2 + W_\e(u) \, dVol_g,
\end{equation}
where $A$ is a bounded open subset of $\Hyp^n$,  $u\in H^1(A;\R)$, and where $\nabla_g$, $\|\cdot\|$  and $dVol_g$ are the gradient, the length of tangent vectors  and the volume element with respect to the Riemannian metric $g$, respectively. 
\par
The critical points, and therefore in particular the minimizers of such energy are solutions of the corresponding Euler-Lagrange equation
\begin{equation}\label{equa}
\Delta_g u + f_\e(u)=0,
\end{equation}
where $f_\e(u):= - W'_\e(u)$ and $\Delta_g$ is the usual (negative) Laplace-Beltrami operator.
\par
In this paper we focus on  entire solutions $u_\e$  (i.e., defined in the whole $\Hyp^n$)  of equation \eqref{equa} that are local minimizers of the energy $\E_\e$ in \eqref{enfun}, according to the following definition. 
\begin{definition}\label{local}
We say that a function $u\in H^1_{loc}(\Hyp^n)$ is a local minimizer of the energy $\E_\e$ defined in \eqref{enfun}, if $\E_\e(u,A)\le \E_\e(v,A)$ for every open bounded subset $A$  of $\Hyp^n$ and for every $v\in H^1_{loc}(\Hyp^n)$ such that  $u-v$ has compact support contained in $A$. 
\end{definition}
In the Euclidean setting, the energy functional \eqref{enfun} is usually referred to as the Cahn-Hillard approximation of the Van der Waals phase transition model. The minimizers  $u_\e$ of the energy $\E_\e$ in \eqref{enfun} describe smeared phase transitions, and their asymptotic behavior as $\e\to 0$ provides a good approximation  of  sharp area minimizing interfaces (actually with constant mean curvature under the usual additional volume constraint).  Roughly speaking, as $\e$ tends to zero,  local minimizers $u_\e$ tend to $\pm 1$ far from a minimal hypersurface $\Sigma$, and make the transition in an  $\e$-neighborhood of their level sets $\Sigma_\e:=\{u_\e = 0\}$, which in turn provide a good approximation of $\Sigma$.   We refer the reader to the important paper \cite{M} for a first rigorous result in this direction and to  \cite{CC}, \cite{HT} for some extensions. Conversely, under suitable non-degeneracy  assumptions, a given  minimal hypersurface $\Sigma$ (or, more generally, a constant mean curvature hypersurface) can be obtained as limit of the zero level sets $\Sigma_\e$ of solutions $u_\e$ to equation \eqref{equa} (see \cite{PR}). 
Thus,  the study of the energy \eqref{enfun} provides a  bridge between semilinear elliptic equations and minimal hypersurfaces both in the Euclidean space and on Riemannian manifolds. We refer the reader to \cite{T2} for a survey on this topic and to \cite{LR} for a first result in case of surfaces of higher codimension related to superconductivity. We refer also to \cite{I} and \cite{BOS} for the analogous link between the gradient flow of (\ref{enfun}) and the mean curvature flow in codimension one and two, respectively.
\par
The goal of this paper is to investigate this classical  connection in hyperbolic space.
The first step in this program  is the construction of entire solutions $u_\e$ to equation \eqref{equa} with prescribed behavior near the {\it sphere at infinity} $\sinf$.  From now on we will assume that  the potential $W_\e$  is of the form $W_\e= \frac{1}{\e^2} W$, where 
$W:\R\to \R_+ \cup\{0\}$ is a $C^2$ function satisfying the following usual assumptions
\begin{equation}\label{ass}
\begin{array}{ll}
i) & W(t) = W(-t) \text{ for all } t\in\R,\\
ii) & \min W = 0 \text{ and }  \{ W(t) = 0\} = \{ -1,1\},\\
iii) & W''(1)>0,\\
iv) & W(t) \text{ is strictly decreasing in } [0,1] \text{ and strictly increasing for } t > 1. 
\end{array}
\end{equation}
We are interested in  solutions which are local energy 
minimizers,   taking the two  minima  $\pm 1$ of the potential  $W$ as boundary values on two different open sets $\Om^{\pm}$ on the sphere at infinity, and making the transition in $\Hyp^n$.  The results we will achieve show that there are plenty of such solutions (actually uncountably many), in analogy with the simpler case of entire bounded harmonic functions in $\Hyp^n$, with arbitrary continuous data at infinity (see \cite{S}, \cite{A3}).  Our results are in striking contrast with what happens in the Euclidean space, where, at least for $n\leq 7$, the nonconstant entire local minimizer of \eqref{enfun} is unique up to isometries (see \cite{SO}, Theorem 2.3). 
\par
Once we have constructed entire solutions $u_\e$ that are locally energy minimizers, the second step consists in letting $\e\to 0$, to obtain a limit function $u^*$ taking only values $\pm 1$ in $\Hyp^n$, in analogy with the asymptotic analysis done in the Euclidean space in \cite{M}. Eventually, when $\Om^+$ and $\Om^-$ have common boundary $L\subset \sinf$, we obtain an existence result for entire minimal hypersurfaces $\Sigma$, the jump set of $u^*$ in $\Hyp^n$, with prescribed behavior $\partial \Sigma = L$ at infinity. This result has been  originally proved in \cite{A1} by methods of geometric measure theory.
\par
For expository convenience  we will state all our results using the Poincar\'e ball model.  First we consider the relevant case where the boundary conditions for $u_\e$ are prescribed on two disjoint spherical caps  $\Om^+ := C^+$ and $\Om^-=C^-$ in $\sinf$ with common boundary $L$.
  
Let $\Sigma= \Sigma(L)$ be the spherical cap in $\Hyp^n$ touching $L$ orthogonally (i.e., the totally geodesic $\Hyp^{n-1}$ asymptotic to $L$ at infinity), and denote by $\dise$ the signed hyperbolic distance from $\Sigma$ with the sign convention $\tilde d(x,\Sigma)\to \pm\infty$ as $x\to C^{\pm}$.  With these special boundary conditions, 
it is possible to perform a one-dimensional reduction of the problem, i.e., to seek for solutions  $U_\e(x):= h_\e (\dise(x,\Sigma(L)))$, for a suitable smooth function $h_\e:\R\to\R$, with $U_\e(x) \to \pm 1$ as $x\to C^{\pm}$ (correspondingly $h_\e(\pm \infty)=\pm 1$).  These kind of one dimensional solutions  exhibit a jump at infinity on the interface $L=\partial C^{\pm}$, that  looks smeared far from the sphere at infinity. The first result of the paper deals with the construction of such one dimensional solutions $U_\e$ that will be the building blocks to treat the general case.  
\begin{theorem}\label{odsthm}
Let  $\Om^+ := C^+$ and $\Om^-:=C^-$ be two disjoint spherical caps  in $\sinf$ with common boundary $L$, and let 
$\Sigma(L)$ be the spherical cap in $\Hyp^n$ touching $\sinf$ along $L$ orthogonally. 
\par
Then there exists a unique solution $U_\e\in C^2(\Hyp^n)$ to equation \eqref{equa} satisfying  $U_\e \equiv 0 $ on $\Sigma$ and   the boundary conditions 
$U_\e(x) \to \pm 1$ as $x\to C^{\pm}$. 
Moreover, $U_\e$ is a local minimizer of the functional $\E_\e$ in \eqref{enfun}, and  it is  one-dimensional, i.e., 
$U_\e (x) = h_\e(\tilde d(x,\Sigma))$ for a suitable smooth odd  increasing function $h_\e$.
Finally, $U_\e (x)\to \text{sgn}(\tilde d(x,\Sigma))$ locally uniformly in $\Hyp^n\setminus \Sigma$ as $\e\to 0$.
\end{theorem}

In dimension two,  existence results for equation \eqref{equa}, with  the potential $W(u)=\frac14 (1-u^2)^2$,  have been largely exploited both in the physical and in the mathematical community, because of its relevance  in the study of the Yang-Mills equations in four dimension. An explicit solution for $\e = 1$ with two point singularities at the boundary has been found in \cite{DFF},  while more   general solutions with two points singularities  have been constructed in \cite{R} and \cite{MZ}, using ODE techniques. A one-dimensional  solution in $\Hyp^n$ for any $n$ has been constructed only very recently in \cite{BM}. The novelty of our result consists in the existence and uniqueness  property for  solutions   vanishing  on $\Sigma$. Clearly, as $\e \to 0 $ the hypersurface $\Sigma$ turns out to be the jump set of the limit function $u^*(x) = \text{sgn}(\tilde d(x,\Sigma))$, thus a totally geodesic and area minimizing hypersurface. As will be clarified  below, the property $U_\e=0$ on $\Sigma$ will be crucial  in order to control the zero level set  of solutions $u_\e$ for general boundary data, and therefore to prescribe the boundary $L$ at infinity of the limiting minimal surface obtained as $\e\to 0$. 
\par
Now we pass to the case of general boundary data, namely to the case of arbitrary open sets $\Om^\pm\subset \sinf$.  In a two dimensional context, a model case is when $\Om^+$ and $\Om^-$ consist in a finite number of arcs. 
The corresponding solutions, usually referred to as {\it multimeron solutions} of the Yang-Mills equations, are solutions of \eqref{equa} with finitely many boundary singularities. They have been conjectured and formally derived  in \cite{W} and \cite{GJ}, and rigorously constructed in \cite{JMZH}, \cite{CGS} and \cite{BPZ}. 
The first  existence result for entire solutions in $\Hyp^n$ with general prescribed behavior on the  sphere at infinity  $\sinf$ is the following.  
\begin{theorem}\label{mainthm}
Let $\Om^+$ and $\Om^-$ be disjoint open subsets of $\sinf$. 
Then, there exists an entire solution $u_\e\in C^2(\Hyp^n)\cap C^0(\Hyp^n\cup\Om^+\cup\Om^-)$ to equation \eqref{equa}, that is a local minimizer of the energy $\E_\e$ in \eqref{enfun} according with Definition \ref{local}, and that satisfies the boundary conditions $u_\e = 1$ on $\Om^+$, $u_\e = -1$ on $\Om^-$. Moreover, the zero level set $\Sigma_\e:= u_\e^{-1}(0)$ satisfies $\Sigma_\e \subseteq \overline{\text{ conv} (F)}$, where $F=\sinf\setminus(\Om^+\cup\Om^-)$.
In addition, $\partial \Om^+ \cap \partial \Om^- \subseteq \overline{\Sigma_\e} \cap \sinf \subset F$ (where  the closure is understood in $\overline{B_1}$ with respect to to the Euclidean topology). In particular, if $\partial \Om^+=\partial\Om^- = F$ then $\overline{ \Sigma_\e} \cap \sinf = F$.
Finally, for $n\le 7$ there exists $\e_0>0$ depending only on $n$, such that for  $\e\le \e_0$ the zero level set $\Sigma_\e$ is a $C^2$ smooth hypersurface.
\end{theorem}

To prove Theorem \ref{mainthm}, the key point is to prescribe the boundary conditions at infinity. To this purpose the main ingredient 
is the construction of suitable barriers $\overline \psi_\e$ and $\underline \psi_\e$ (so that $\underline \psi_\e\le u_\e \le \overline \psi_\e$) with desired behavior at infinity, obtained combining the one dimensional  solutions discussed above. It turns out that the location  of the zero level set $\Sigma_\e$ is controlled by the barriers, so that it is trapped into $\overline{\text{ conv} (F)}$, the geodesic convex hull of $F$ in $\Hyp^n \cup \sinf$. On the other hand, smoothness of $\Sigma_\e$ in low dimension is indeed a consequence through blow-up analysis of the recent important paper \cite{SO}. This existence result combines ideas from \cite{CGS} and \cite{JMZH}, and gives a positive answer to the question, raised in \cite{BM},  of constructing entire solutions to equation \eqref{equa}, taking values in $\{-1,0,+1\}$ on prescribed sets of the sphere at infinity $\sinf$.   
\par
Our next result deals with solutions exhibiting a prescribed  sharp interface $L\subset \sinf$.  To this purpose, given $B\subset \sinf$,  we denote by  $K(B)$ the cone over $B$ from the origin in the Poincar\'e ball model, defined by 
$K(B):= \cup_{\rho < 1} \rho B \cup \{0\}$.  Moreover, $\dise(x, K(L))$ will denote now the hyperbolic signed distance function from $K(L)$, taking positive sign in $K(\Om^+)$ and negative sign in $K(\Om^-)$. 

\begin{theorem}\label{ab}
Let $\Om^+$ and $\Om^-$ be disjoint open subsets of $\sinf$ with common boundary $L$, and assume that $L\subset \sinf$  is a smooth hypersurface of class $C^1$.
\par
Then, there exists an entire solution $u_\e\in C^2(\Hyp^n)\cap C^0(\Hyp^n\cup\sinf\setminus L)$ to equation \eqref{equa}, that satisfies the boundary conditions $u_\e = 1$ on $\Om^+$, $u_\e = -1$ on $\Om^-$,   that is a local minimizer of the energy $\E_\e$ in \eqref{enfun}, and having the following asymptotic behavior near the sphere at infinity $\sinf$
\begin{equation}\label{abe}
u_\e(x) = h_\e(\dise(x,K(L))) + e(x), \qquad \text{ where } e(x)\to 0 \text{ as } x\to \sinf. 
\end{equation}
Moreover, the zero level set $\Sigma_\e:= u_\e^{-1}(0)$ satisfies $\Sigma_\e \subset \overline{\text{ conv} (L)}$. Finally, $\Sigma_\e$   is a $C^1$   hypersurface near the sphere at infinity with boundary $\partial \Sigma_\e = L$,  touching $\sinf$  orthogonally along $L$. 
\end{theorem}
The asymptotic expansion in \eqref{abe}  generalizes the analogous  property established in \cite{CGS} for solutions near isolated singularities in dimension  two (with $\e=1$ and the explicit potential $W(u)=\frac14(1-u^2)^2$).
Indeed, as we will see  in Proposition \ref{blowup},   blowing up  the solution $u_\e$ 
around a point of $L$  the sets $\Om^\pm$ converge (under rescaling) to a pair of half spheres, while 
$u_\e$ converges to the corresponding one dimensional solution given by Theorem \ref{odsthm}, and this will be the key step in proving \eqref{abe}.
\par
The orthogonality of the zero level set of the minimizers $u_\e$ at the boundary, stated in Theorem \ref{ab},  means that the normals $\nu_{\Sigma_\e}(P_k)$ at points  $P_k\in \Sigma_\e$ converging to some  $P_\infty \in L\subset \sinf$ as $k \to \infty$,  tend to $\nu_L(P_\infty)$ (see \eqref{orto}).  This orthogonality property
can be seen as the natural  counterpart, in this phase field framework, of the boundary orthogonality proved in \cite{HL} for entire minimal hypersurfaces constructed in  \cite{A1}. For the reader convenience we quote  this last result in Theorem \ref{micorre}.  In fact, in \cite{HL} the authors actually give a complete boundary regularity result, that could be interesting to exploit in our context, in order to obtain higher regularity of $\Sigma_\e$ when  $L$ is more regular than $C^1$. To this purpose, it seems very natural to investigate the asymptotic behavior of the solutions $u_\e$ near the sphere at infinity, either through a PDE approach or through an asymptotic energy expansion based on $\Gamma$-convergence, but  we will not pursue further  this point in the paper.
\par
Let us pass now to describe the second step of our program, consisting in letting $\e\to0$, recovering in the limit sharp area minimizing interfaces in hyperbolic space. 
 The language of $\Gamma$-convergence, as shown in \cite{M} in the Euclidean case, provides the natural framework to perform this asymptotic analysis (we refer the reader  to the book \cite{DM} for an extensive introduction to the subject). 
\par
Since we are interested in minimal hypersurfaces in hyperbolic space with infinite area,
 in order to perform our variational approach based on $\Gamma$-convergence it is convenient to restrict the energy functionals  to  bounded domains of $\Hyp^n$. More precisely, we identify $\Hyp^n$ with $B_1$ according with the Poincar\'e ball model, and  we  restrict the energy functionals $\E_\e(u)$ defined in \eqref{enfun} to balls $B_R$ with $0<R<1$. Moreover,   we  fix a boundary condition $u\equiv w_\e$
 on $\partial B_R$, where $w_\e$ belongs to $H^1_{loc}(B_1)$ with $|w_\e|\le 1$,  having in mind $w_\e = u_\e$ for our purposes, where $u_\e$ is the local minimizer constructed in Theorem \ref{mainthm}.
 \par 
Let $\epsm\to 0$,
and  assume that (up to subsequences)
\begin{equation}\label{datalbo}
w_\epsm \to w^* \text{ in } L_{loc}^1(B_1), \quad
\mu_\epsm: = \epsm \big( \frac{1}{2} \|\nabla w_\epsm\|^2 + W_\epsm(w_\epsm)\big) \, dVol_g \weakstar \mu^*,
\end{equation}
for some $w^*\in BV_{loc}(B_1;\{+1,-1\})$ and for some locally finite positive  measure $\mu^*$ on $B_1$.
\par
We are in a position to define 
 the energy functionals $\F_\e(\,\cdot\,; w_\e, B_R):L^1(B_R)\to \R$ as follows
\begin{equation}\label{feps}
\F_\e(u; w_\e, B_R):= 
\begin{cases}
\sqrt{2} \e \,  \E_\e(u,B_R) & \text{ if } u\in H_{ w_\e}^1(B_R),\\
\infty & \text{ otherwise in }  L^1(B_R),
\end{cases}
\end{equation}
where $H_{ w_\e}^1(B_R)$ is the set of $H^1$ functions with trace on $\partial B_R$ equal to $w_\e$. 
Note that $\F_\e$ is lower semicontinuous even if it could be infinite on some $u \in H_{ w_\e}^1(B_R)$, because we impose no growth condition on $W_\e$ at infinity.

Given $v\in L^1(B_R)$,  denote by $\tilde v= \tilde{v}_{{w^*}}$ the extension of $v$ to $B_1$, coinciding with $w^*$ on $B_1\setminus B_R$.
The candidate $\Gamma$-limit of the functionals $\F_\e$ is the functional $\F(\,\cdot\,;w^*,B_R):L^1(B_R)\to \R$ defined as
\begin{equation}\label{gali}
\F(v;w^*,B_R):= 
\begin{cases}
C_W  |\tilde  v_{w^*}|_{BV_g(\overline B_R)} & \text{ if } v \in BV(B_R;\{+1,-1\});\\
+\infty & \text{ otherwise in } L^1(B_R),
\end{cases}
\end{equation}
where 
$$
C_W = \int_{-1}^{1} \sqrt{W(s)} \, ds,
$$
and $|\cdot |_{BV_g}$ denotes the intrinsic total variation in the hyperbolic space (see Section \ref{preliminary} for the precise definition). 
\par
The following result describes the asymptotic behavior of the energy functionals $\F_\e$ as $\e\to 0$. 
\begin{theorem}\label{gcthm}
Let $\epsm\to 0$, and let  $w_\epsm$ be a sequence of boundary conditions  satisfying  \eqref{datalbo} for some suitable $w^*$ and  $\mu^*$.  The following compactness and $\Gamma$-convergence result holds.
\begin{itemize}
\item[i)] {\rm (Compactness.)} Let $0<R<1$ be fixed, and let $v_\epsm$ be a sequence in $L^1(B_R)$ with $|v_\epsm| \le 1$ such that  $\F_\epsm(v_\epsm;w_\epsm,B_R)\le C$, for some constant $C$ independent of $\epsm$. 
Then  (up to a subsequence)  $v_\epsm\to v^*$ in $L^1(B_R)$ for  some  $v^* \in BV(B_R;\{+1,-1\})$.
\item[ii)] {\rm ($\Gamma$-convergence.)} 
Let $0<R<1$ be such that $\mu^*(\partial B_R)=0$. Then the following  $\Gamma$-convergence inequalities hold.
\begin{itemize}
\item[i)] {(\rm  $\Gamma$-liminf inequality.)} Let $v_\epsm \to v$  in $L^1(B_R)$. Then we have $\F(v; w^*,B_R)\le \liminf_\epsm \F_\epsm(v_\epsm;w_\epsm,B_R)$;
\item[ii)] {(\rm  $\Gamma$-limsup inequality.)} Let $v\in L^1(B_R)$. Then  there exists a sequence $v_\epsm\to v$  in $L^1(B_R)$ such that $\F(v;w^*,B_R)\ge \limsup_\epsm \F_\epsm(v_\epsm;w_\epsm,B_R)$. 
\end{itemize}
\end{itemize}
\end{theorem}
This result represents the counterpart in the hyperbolic space of the classical $\Gamma$-convergence result  \cite{M}  for phase transitions in the Euclidean space. 
As for the Euclidean setting \cite{M},  the bound $|v_\epsm| \le 1$ in the compactness  statement is  very natural because the energy functionals decrease under truncation, and  it can be dropped assuming 
super-quadratic  growth conditions on the potential $W$ at infinity, like in the model case $W_\e(u)= \frac{1}{4 \e^2} (1-u^2)^2$. Note also that the $\Gamma$-limsup inequality would fail if in \eqref{gali} we neglect the contribution due to the possible jump between $w^*$ and $v$ across $\partial B_R$. In addition, the fact that this boundary contribution depends only on $w^*$ is indeed a consequence of the assumption $\mu^*(\partial B_R) = 0$ (see Remark \ref{assmuzero}). 
The previous  $\Gamma$-convergence result, applied to the local minimizers $u_\e$ yields the following theorem.

\begin{theorem}\label{micorre}
Let $\Om^+$ and $\Om^-$ be disjoint open subsets of $\sinf$, and let $F:=\sinf\setminus (\Om^+\cup\Om^-)$. 
Let $\epsm\to 0$ and let $u_\epsm$ be the locally minimizing  entire solutions of \eqref{equa} given by Theorem \ref{mainthm}   and Theorem \ref{ab}.
Then the following holds. 
\begin{itemize}
\item[i)] Up to a subsequence, $u_\epsm\to u^*$ in $L^1_{loc}(B_1)$ for some $u^*\in BV_{loc}(B_1;\{-1,1\})$. Moreover, the jump set  $ S_{u^*}$ satisfies  
$ S_{u^*}\subset \overline{\text{conv}(F)}$, and  
$\partial \Om^+ \cap \partial \Om^- \subseteq \overline{ S_{u^*}} \cap \sinf \subset F$ (where  the closure is understood in $\overline{B_1}$ with respect to to the Euclidean topology). In particular, if $\partial \Om^+=\partial\Om^- = F$ then $\overline{ S_{u^*}} \cap \sinf = F$.
\item[ii)] The limit $u^*$ is a local minimizer of the total variation, i.e., 
$|u^*|_{BV_g(B_R)} \le |v^*|_{BV_g(B_R)}$ for every $v^*\in BV_{loc}(B_1;\{+1,-1\})$ such that the support  of $(u^*-v^*)$ is compactly contained in some ball $B_R$, $0<R<1$.
\item[iii)]
The $(n-1)$-current $J_{u^*}$ corresponding to the jump set $S_{u^*}$  is a local mass minimizer, therefore for $n\leq 7$ it is a smooth (analytic) hypersurface, while for $n>7$  it has  a singular set $Z$ of dimension $dim \, Z \leq n-8$.  Finally, if $L:=\partial \Om^+ = \partial \Om^- = F$ is a $C^1$ hypersurface, then $J_{u^*}$ is a smooth hypersurface near the sphere at infinity, touching $\sinf$ orthogonally along $L$.
\end{itemize}
\end{theorem}
Compactness of local minimizers follows from $\Gamma$-convergence, while the behaviour of the barriers as $\epsm \to 0$ allows to use  all the information on the zero sets $\Sigma_\epsm$ and to control the position of the jump set $S_{u^*}$ and its behaviour at infinity.   The minimality property of $u^*$  is a direct consequence of the fact that the minimality of $u_\e$ passes to the limit under $\Gamma$-convergence. The last part of the theorem is essentially well known, so we include it just for reader convenience. Indeed, the minimality for the current $J_{u^*}$ corresponding to $S_{u^*}$ is standard and  its interior regularity is a consequence of the celebrated regularity results for codimension-one mass minimizing currents (see \cite{F}). On the other hand, the last statement concerning boundary regularity and orthogonality at infinity has been established in \cite{HL}.

When $\Om^+\cup\Om^-$ is a dense open subset of $\sinf$,  the existence of a minimal hypersurface asymptotic to $F$ at infinity was originally proved in \cite{A1} for $F=\partial \Omega^+=\partial \Omega^-$ and $F=L$ an  immersed  smooth hypersurface, while for very irregular (possibly fractal) interfaces $F=\partial\Om^+=\partial \Om^-$ the result has been proved in \cite{La}. Here we consider a more general case without assuming 
$\Om^+\cup\Om^-$  dense. As a consequence the hypersurface is hinged at infinity only on the contact region $\partial\Om^+\cap \partial \Om^-$, while we expect that $J_{u^*}$ (and indeed also its boundary at infinity)  is a minimizer of a suitable free boundary problem.

Now we would like to discuss few possible directions of investigation. In our opinion it would be interesting to extend the phase transition approach to the case of constant mean curvature hypersurfaces with prescribed asymptotic boundary, as constructed in \cite{T} and \cite{GS} working with finite perimeter sets, and with the prescribed mean curvature equation respectively. On the other hand, another direction of investigation could be, in the same spirit of \cite{A2},  to study minimizing solutions to \eqref{equa} on hyperbolic manifolds, i.e. to investigate entire solutions to \eqref{equa} which are invariant under some discrete cocompact subgroup of isometries in $\Hyp^n$. It would be interesting as well to push further this method in order to deal with the vector valued case. In this way, one could obtain  minimal surfaces of higher codimension with prescribed behavior at infinity (already constructed in \cite{A1} using geometric measure theory) as a limit of solutions of elliptic systems. In the Euclidean framework the picture is quite well developed; the $\Gamma$-convergence result  has been done in \cite{ABO}, while for the asymptotic analysis of minimizers in the codimension-two case we refer to \cite{LR}. 
\par Finally, we mention that for $n=2$ a discrete analogue of our problem is given by the Ising model on hyperbolic graphs (i.e. on Cayley graphs corresponding to discrete cocompact groups of isometries acting on the hyperbolic plane) considered e.g. in \cite{SS}. For this model, we expect existence of uncountably many distinct local minimizers of the Hamiltonian which should be the natural discrete counterpart of the ones given by Theorem \ref{mainthm}. The presence of several local minimizers would be consistent with the existence of uncountably many mutually singular Gibbs measures on the the set of all spin configurations, rigorously proved in \cite[Theorem 1]{SS}, for sufficiently high inverse temperature. 

\section{Preliminary overview on the hyperbolic space}\label{preliminary}
In this section we will briefly review  the hyperbolic space, described according with  the {\it half space model} and the {\it Poincar\'e ball  model}. For each of these models, we recall the corresponding  metric, the volume element, the geodesics, and the notion of sphere at infinity. We introduce in these models our  energy functional and  
the corresponding Euler-Lagrange equation. We do not review the description of the group of isometries of each model in terms of their conformal homeomorphisms and instead we refer the interested reader e.g. to \cite{B}, Chapter 3. Finally we recall the basic definitions of BV functions on the hyperbolic space that we will need in the last section of the paper.   

\subsection{The half space model}
In this model,  the hyperbolic space $\Hyp^n$ is given by the half space
$$
\R^n_+ := \{(x_1, \ldots, x_n) \in \R^n: \, x_n>0 \}, 
$$ 
endowed with the Riemannian metric 
$$
g:= \frac{\sum_{i=1}^n dx_i^2}{x_n^2}.
$$
The induced volume element is given by
$$
dVol_g:= \frac{dx}{x_n^n}, 
$$
where $dx$ denotes the usual Lebesgue measure in $\R^n$. 
\par
The compactification of  the hyperbolic space is obtained adding to $\Hyp^n$ the so called {\it Sphere at infinity}  $\sinf$, that in  the half space model is given by 

$$
\sinf:= \partial \R^n_+ \cup \{\infty\}.
$$
Given two points $p$ and $q\in \Hyp^n$, the  geodesic joining $p$ and $q$ is given by  an arc of circle or by a segment  (joining $p$ and $q$), contained  in the only semi-circle or half line through  $p$ and $q$  and touching the hyper-plane $\partial \R^n_+$ orthogonally. 
\par
Finally, since $\nabla_g u(x) = x_n^2 \nabla u(x)$,  the energy functional \eqref{enfun} can be rewritten more explicitly as  
\begin{equation} \label{enfun2}
\E_\e(u,A):= \int_A \Big( \frac{1}{2} x_n^2 |\nabla u|^2 + W_\e(u)\Big) \, \frac{dx}{x_n^n},
\end{equation}
while, recalling that $f_\e(s) = -W'(s)$,  the corresponding Euler-Lagrange equation  \eqref{equa} reads as
\begin{equation}\label{equa2}
x_n^2 \Delta u + (2-n) x_n \partial_{x_n} u + f_\e(u)=0.
\end{equation}

\subsection{The Poincar\'e ball model}
In this model the hyperbolic space $\Hyp^n$ is given by the unit ball
$$
B_1:= \{x \in \R^n: \, |x|<1  \}, 
$$ 
endowed with the Riemannian metric 
$$
g:= \frac{4\sum_{i=1}^n dx_i^2}{(1-|x|^2)^2}.
$$
The corresponding volume element is given by
$$
dVol_g:= \frac{2^n dx}{(1-|x|^2)^n}.
$$

The  Sphere at infinity $\sinf$ in this case is just given by $\partial B_1$. 
Moreover, given two points $p$ and $q\in \Hyp^n$, the  geodesic joining $p$ and $q$ is given by  an arc of circle or by a segment  (with extremes $p$ and $q$), contained  in the only circle or  chord  passing through  $p$ and $q$  and touching $\partial B_1$  orthogonally. 
\par
Finally,  since $\nabla_g u(x) = \frac{(1-|x|^2)^2}{4}  \nabla u(x)$, the energy functional \eqref{enfun} is given by  
\begin{equation} \label{enfun3}
\E_\e(u,A):= \int_A \Big( \frac{1}{8} (1- |x|^2)^2 |\nabla u|^2 + W_\e(u)\Big) \, \frac{2^n dx}{(1-|x|^2)^n},
\end{equation}

while the corresponding Euler-Lagrange equation  \eqref{equa} reads as

\begin{equation}\label{equa3}
\frac{(1-|x|^2)^n}{2^n} \text{div} \left( \Big(\frac{1-|x|^2}{2}\Big)^{2-n} \nabla u\right) +f_\e(u)=0.
\end{equation}

\subsection{BV functions in $\Hyp^n$}
For the general theory of functions of bounded variation, 
we refer to the standard reference monograph \cite{AFP}, and we refer to \cite{MPPP} for the theory on Riemaniann manifolds ; here
we recall some basic definitions and properties we need in the sequel, confining ourselves  to  $BV$ functions defined on the hyperbolic space $\Hyp^n$. 
\par
Given any open set $A\subset\subset\Hyp^n$ compactly contained in $\Hyp^n$, we recall that  $u \in BV_g(A)$ if $u \in L^1(A, dVol_g)$,  and  it has finite total variation $|u|_{BV_g(A)}$, where  
$$
|u|_{BV_g(A)}:= \sup\left\{ \int_A u \text{ div}_g \Phi \, dVol_g, \,\Phi\in C^\infty_0(A;TA), \,  \|\Phi\| \le 1\right\} <\infty.
$$
Note that, since the hyperbolic metric is locally equivalent to the Euclidean one, we have (for any model of $\Hyp^n$) $BV_g(A) = BV(A)$ with equivalent, but not identical norms. As for the Euclidean case, we say that $u\in BV_{g, loc}(\Hyp^n)$ if $u$ (restricted on $A$) belongs to $BV_g(A)$ for every  open set $A$ compactly contained in $\Hyp$. 
In this case it turns out that the jump set $S(u)$, i.e., the set  of points $x \in A$ 
which are not Lebesgue points of $u$  (also referred to as the singular set of $u$), is $(n-1)$-rectifiable, that is
there exists a sequence of $C^1$ hypersurfaces $(M_i)_{i \in \N}$
such that $S(u) \subseteq \cup_i M_i$ up to a set of
$\hs^{n-1}${-}measure zero.
\par
We are interested in functions  $u \in BV_{loc, g}(\Hyp^n;\{-1,+1\})$, i.e., functions  $u \in BV_{g,loc}(\Hyp^n)$ valued in $\{-1,+1\}$.
For such functions we denote by $|D_gu|(A) =|u|_{BV_g(A)}$ the total variation of $u$ on $A$. It turns out that $|D_gu|(\cdot)$ is a locally finite Borel measure on $\Hyp^n$, and the following representation formula holds 
\begin{equation}
\label{hyptotvar}
|D_g u|(A)= 2 \hs_g^{n-1}(S(u)\cap A)  \qquad \text{ for all open set }
A \subset\subset \Hyp^n,
\end{equation}
where  $\hs_g^{n-1}$ denotes the $(n-1)$-dimensional Hausdorff  measure associated to the hyperbolic distance on $\Hyp^n$.  
Notice that, in the half space model, we have 
\begin{equation}
\label{halfsptotvar}
|D_g u|(A)= 2 \int_{S(u)\cap A}  \frac{1}{x_{n}^{n-1}} d\hs^{n-1},
\end{equation}
where $d\hs^{n-1}$ denotes now the standard Euclidean $(n-1)$-dimensional Hausdorff measure, thus it is the usual Euclidean formula up to a conformal factor due to the hyperbolic metric. 

\section{One-dimensional phase transitions}
\label{odsSec}
In this section we will construct elementary solutions to equation \eqref{equa}, i.e., solutions corresponding  to the case when $\Om^+$ and $\Omega^-$ are  disjoint spherical caps with common boundary. 
\par
We will  work mainly in the half space model, where we construct  elementary solutions through a one dimensional reduction, then solving an ODE in $\R$ by a minimization argument, in the spirit of \cite{BM}. Our method will produce in particular odd solutions $h_\e$, and this property will be essential in our approach, since it provides  the desired asymptotic behaviour   as $\e\to 0$ of the barriers $\underline\psi_\e$, $\overline\psi_\e$ that we will construct in Section \ref{mpt}.  Moreover, we give a uniqueness result for solutions of the ODE vanishing at zero, which in turns yields the uniqueness property for elementary solutions vanishing on $\Sigma$.

\subsection{One-dimensional reduction and existence for the ODE}\label{ods}
Here we are looking for particular elementary solutions  $u(x_1, \ldots, x_n)$ to equation \eqref{equa2}, which are odd with respect to $x_1$, and satisfying the boundary condition  
$u(x) = \text{sgn}(x_1)$ on the hyperplane $\{x_n =  0\}$. More precisely, we construct one dimensional solutions, which are constant on the level sets  $\{\frac{x_1}{x_n} = c\}$ of the  distance function from  $\Sigma_0:=\{x_1 = 0\}\subset \R^n_+.$ Thus,  
enforcing that the solution takes the form 
$$
u_\e(x) = g_\e\Big(\frac{x_1}{x_n}\Big),
$$
we obtain the following boundary value problem for $g_\e(\xi)$,\, $\xi=\frac{x_1}{x_n}$\,,
\begin{equation}\label{eq0dim}
\begin{cases}
(1 + \xi^2) g_\e''(\xi) + n \xi g_\e'(\xi) =  - f_\e(g_\e(\xi));\\
g_\e(\pm \infty) = \pm 1.
\end{cases} 
\end{equation}
Since  the signed distance $\tilde d$ from $\Sigma_0$ satisfies $\tilde d(x,\Sigma_0) = \sinh^{-1}(\xi)$, it is convenient to 
set $\tau= \sinh^{-1} (\xi)$ and to define $h_\e(\tau) = g_\e(\xi)$, so that 
$h_\e$ has to solve    
\begin{equation}\label{eq1dim}
\begin{cases}
h_\e''(\tau) + (n-1) \tanh \tau  h_\e'(\tau) = - f_\e(h_\e(\tau));\\
h_\e(\pm \infty) = \pm 1.
\end{cases} 
\end{equation}

Such equation is the Euler-Lagrange equation of the energy functional 

\begin{equation}\label{1dene}
E_\e(h) = \int \Big(\frac{1}{2} {h'}^2 + W_\e(h)\Big) \cosh^{n-1}\tau \, d\tau.
\end{equation}

\begin{proposition}\label{odsprop}
Let $\e>0$ and let $f_\e= - W'_\e$, with $W$ satisfying the assumptions in \eqref{ass}. Then problem \eqref{eq1dim} admits a solution $h_\e$ which is odd and strictly increasing.  
\par
Moreover, $E_\e(h_\e)\le C/\e$ for some positive constant $C$ independent of $\e$, and $h_\e(\tau)\to \text{sgn}(\tau)$ locally uniformly in $\R\setminus \{0\}$ as $\e \to0$.
\end{proposition}
\begin{proof}
In order to find a solution of equation \eqref{eq1dim}, we consider the following minimization problem
\begin{equation}\label{mp}
\min \{ E_\e(h), \, h\in H_{loc}^1(\R_+), \, h(0)=0  \}.
\end{equation}
Let as first prove that the minimum problem \eqref{mp} admits a minimizer $h^+_\e$ which is increasing and satisfies $h_\e(+\infty)= 1$. 
\par
Let $h^+_{\e,k}$ be a minimizing sequence for \eqref{mp}. Since the potential $W$ is even, we may assume without loss of generality (taking the absolute value if necessary)  that $h^+_{\e,k}$ are positive.  
Since $h^+_{\e,k}$ have finite energy and $h^+_{\e,k}(0)= 0$ we easily deduce that  $h^+_{\e,k}$ is bounded in  $H_{loc}^1(\R_+)$, i.e., it is bounded in $H^1(0,M)$ for every positive $M$. Therefore, in view of  the compact embedding $H^1_{loc}\hookrightarrow C^0_{loc}$,  a diagonal argument yields  that (up to a subsequence)  $h^+_{\e,k}$  converges locally uniformly to some continuous function $h^{+}_\e$,  with $h^{+}_\e(0)=0$, and  $h^+_{\e,k} \weak  h^{+}_\e$ in $H^1(0,M)$ for every $M$, so that in particular  $h^{+}_\e$ belongs to $H_{loc}^1(\R_+)$. Since for all positive $M$ the 
functional $ E_\e$ is weakly lower semicontinuous in $H^1(0,M)$, and since $h^+_{\e,k}$ is a minimizing sequence, 
we have
\begin{multline}\label{lsc}
\int_0^M \Big(\frac{1}{2} {(h^+_\e)'}^2 + W_\e(h^+_\e)\Big) \cosh^{n-1}\tau \, d\tau 
\\
\le \liminf_k
\int_0^M \Big(\frac{1}{2} {(h^+_{\e,k})'}^2 + W_\e(h^+_{\e,k})\Big) \cosh^{n-1}\tau \, d\tau
\le \lim_k E_\e(h^+_{\e,k}) = \inf E_\e. 
\end{multline}

Since $h^+_\e$ is an admissible function in the minimum problem \eqref{mp}, passing to the limit for $M\to\infty$ in \eqref{lsc} we conclude that
$h^+_\e$ is a minimum point. Clearly $h_\e^+ \not \equiv 0$ and since it has finite energy, by  a truncation argument  we also deduce that $0 \le h^+_\e \le 1$ and that $\limsup_\tau h_\e^+=1$ as $\tau \to \infty$. 
\par
Setting $w_\e (\tau):= \min\{ \tau/\e  , 1\}$ we have that   $E_\e(w_\e)\le C \e^{-1}$ for some $C>0$ independent of $\e$, and therefore 
\begin{equation}\label{unili}
E_\e(h^+_\e)\le E_\e(w_\e)\le \frac{C}{\e}.
\end{equation}
Let us prove by a contradiction argument that $h^+_\e$ is non decreasing.     
Since $h_\e^+$ is continuous, $0 \leq h_\e^+\leq 1$ and $\limsup_\tau h_\e^+=1$ as $\tau \to\infty$,  we may assume by contradiction that there exist three points $\tau_1 < \tau_2 < \tau_3$ with $0 \leq h^+_\e(\tau_2 ) < h^+_\e(\tau_1 )= h^+_\e(\tau_3) <1$. 
Set $w(\tau) = \max\{h^+_\e(\tau) , h^+_\e(\tau_3) \}$ for every $\tau \in (\tau_1,\tau_3)$. Then, replacing $h^+_\e$ with $w$ in $(\tau_1,\tau_3)$ we obtain an admissible function $\tilde h$, with $E_\e(\tilde h)< E_\e(h^+_\e)$, which is in contradiction with the minimality of $h^+_\e$.
Now we claim that we have $(h^+_\e(\tau))'>0$ in $[0,\infty)$. Indeed, if $\tau = 0 $ then  $(h^+_\e(\tau))'>0$ because otherwise we would have by ODE uniqueness   $h^+_\e\equiv 0$, which  contradicts $E_\e(\tilde h^+_\e)<\frac{C}{\e}$. While for $\tau > 0$,  if $(h^+_\e(\tau))'=0$ equation \eqref{eq1dim} would  imply $(h^+_\e(\tau))''<0$, which is in contradiction with the monotonicity of $h_\e^+$ because $(h^+_\e )'$ would be negative just after $\tau$.
\par
By the fact that  $h^+_\e$ is bounded and increasing we deduce that it admits limit for $\tau \to \infty$; moreover   $h^+_\e(+\infty)=1$ because  $E_\e(h^+_\e)<C/\e$ and $W(t) = 0$ exactly on $\{t= \pm 1\}$. Analogously, since  $E_\e(h^+_\e)<C/\e$  it also follows that $h^+_\e$ converges to $1$ locally uniformly in $\R_+$.
\par
Finally, we define $h_\e(\tau)$ as the odd reflection of $h^+_\e$, i.e., $h_\e(\tau):= \text{sgn}(\tau) h^+_\e(|\tau|)$, hence $h_\e\in C^1(\R)$. By minimality we have that  $h^+_\e$  solves the  equation  in $(0,+\infty)$, and therefore $h_\e$ is a solution of  \eqref{eq1dim}, and it has all the desired properties.  
\end{proof}

\subsection{A uniqueness property for solutions of the ODE}
In this paragraph we  provide a variational characterization for the solutions of  problem  \eqref{eq1dim}. As a consequence   we obtain  that there exists a unique solution vanishing at zero. We start with the following lemma.

\begin{lemma}\label{lemmasi}
Every solution to  problem \eqref{eq1dim} is strictly increasing and has finite energy.  
\end{lemma} 

\begin{proof}
Let $k_\e$ be a solution of \eqref{eq1dim}. In order to prove that it is strictly increasing, it is clearly enough to show that $k_\e' (\tau)\neq 0$
for every $\tau\in\R$. To this purpose, set 
$V(h,h'):= 1/2  (h'(\tau))^2 - W(h)$.  It is easy to see that 
\begin{equation}\label{liapu}
\frac{d}{d\tau} V(k_\e(\tau), k_\e' (\tau)) = - (n-1) \tanh \tau (k_\e'(\tau))^2.
\end{equation}
Assume by contradiction that  $k_\e' (\bar\tau) = 0$ for some $\bar\tau\in\R$. Then we clearly have $|k_\e(\bar \tau) |\neq 1$ by ODE uniqueness, and therefore, $V(k_\e(\bar\tau), k_\e' (\bar\tau)) =\alpha < 0$. We consider only the case $\bar\tau>0$, the other case being analogous. By \eqref{liapu} we deduce  $V(k_\e(\tau), k_\e' (\tau)) \le \alpha <0$ for every $ \tau\ge \bar\tau$, which clearly gives a contradiction since  $\liminf_{\tau\to\infty} V(k_\e(\tau),k_\e(\tau)') \ge \liminf_{\tau\to\infty} -W(k_\e(\tau)) = 0$.
\par 
Let us prove now that $k_\e$ as finite energy. Multiplying both sides of \eqref{eq1dim} by $k_\e' \cosh^{n-1}\tau$ we have 
\begin{equation}\label{stieneloc}
\frac{d}{d\tau} \left(\frac{1}{2} (k_\e')^2 \cosh^{n-1}\tau\right) + \frac{n-1}{2} (k_\e')^2 \cosh^{n-1} \tau \tanh \tau + f_\e(k_\e) k_\e' \cosh^{n-1} \tau = 0.
\end{equation}
Integrating equation \eqref{stieneloc} between  $\bar \tau:= k_\e^{-1}(0)$ and  $\tau \in \R$, we deduce
\begin{multline}\label{multi}
 \left(\frac{1}{2} (k_\e'(\tau))^2 \cosh^{n-1}\tau\right) + \int_{\bar \tau}^\tau  \left( \frac{n-1}{2} (k_\e'(s))^2   \tanh s + f_\e(k_\e(s)) k_\e^\prime(s) \right) \cosh^{n-1} s \, ds =\\
=  \left(\frac{1}{2} (k_\e'(\bar\tau))^2 \cosh^{n-1}\bar\tau\right) . 
\end{multline}
Since $f_\e(t)t\ge 0$ for $|t|\le 1$ and $k_\e'>0$, we easily obtain that for every $\tau \in \R$
$$
 \left(\frac{1}{2} (k_\e'(\tau))^2 \cosh^{n-1}\tau\right)
 + \int_{\bar \tau}^\tau  f_\e(k_\e(s)) k_\e^\prime(s)  \cosh^{n-1} s \, ds \ge 0.
$$
Since  $\tanh(s)\to \pm 1$  as $s\to \pm \infty$, \eqref{multi} yields
\begin{equation}\label{cinepot}
\int_\R \frac{1}{2} (k_\e'(\tau))^2 \cosh^{n-1}(\tau) \, d\tau < \infty,
\qquad
(k_\e'(\tau))^2 \cosh^{n-1}(\tau)\le C < \infty \text{ for every } \tau \in \R. 
\end{equation}
By \eqref{multi} and \eqref{cinepot} we easily deduce that
\begin{equation}\label{inter}
\left|\int_{\bar \tau}^\tau  f_\e(k_\e(s)) k_\e'(s)  \cosh^{n-1} s   \, ds \right| \le C \qquad \text{ for a constant $C$ independent of } \tau. 
\end{equation}
A simple integration by parts gives 
\begin{multline}\label{ulti}
\int_{\bar \tau}^\tau  f_\e(k_\e(s)) k_\e'(s)  \cosh^{n-1} s   \, ds  = 
\\
=
\int_{\bar \tau}^\tau (n-1) W_\e(k_\e(s)) \cosh^{n-1}s \tanh s - \frac{d}{ds} (W_\e(k_\e(s)) \cosh^{n-1} s)  \, ds =
\\
= \int_{\bar \tau}^\tau (n-1) W_\e(k_\e(s)) \cosh^{n-1}s \tanh s \, ds + W_\e(k_\e(\bar \tau)) \cosh^{n-1} \bar \tau   
- W_\e(k_\e( \tau)) \cosh^{n-1}  \tau 
\end{multline}
Taking into account the exponential decay of $k_\e'$ given by the second inequality in \eqref{cinepot}, a simple integration yields the exponential decay of $1- k_\e^2(\tau)$, so that $(1-k_\e^2(\tau))^2 \cosh^{n-1} \tau \le C <\infty$ for every $\tau$. By Taylor expansion around the minima of $W$ we also get 
\begin{equation}\label{po}
W_\e(k_\e(\tau)) \cosh^{n-1} \tau \le C <\infty
\qquad \text{ for every } \tau. 
\end{equation}
By \eqref{inter}, \eqref{ulti} and \eqref{po} we deduce that 
$$
\int_\R  W_\e(k_\e) \cosh^{n-1}(\tau) \, d\tau < \infty,
$$
which together with \eqref{cinepot} yields  $E_\e(k_\e)<\infty$.  
\end{proof}

\begin{proposition}\label{nac}
Every solution $k_\e$ to \eqref{eq1dim} is strictly increasing and minimizes  the energy $E_\e$ in \eqref{1dene} 
among all smooth functions $h$ satisfying $h(\pm \infty) = \pm1$. 
\par
As a consequence,  the solution $h_\e$ provided by Proposition \ref{odsprop} is the unique  solution to \eqref{eq1dim} vanishing at zero.
\end{proposition}

\begin{proof}
In view of Lemma \ref{lemmasi} we have that $k_\e$ is strictly increasing and it has finite energy.  
Now we show that $k_\e$ is the unique energy minimizer in every compact interval $I\subset \R$, with respect to its own boundary values. We will use a contradiction argument similar to the one in the proof of \cite[Theorem 2.3]{R}. Assume that there exists an energy minimizer $j_\e \neq k_\e$  (with the same boundary  values),  set 
$
A_1 := \{ j_\e> k_\e\},$\,    $A_2 := \{ j_\e< k_\e\},
$  
and let us show that both these open sets are empty. We  show only that $A_1= \emptyset$, 
since $A_2= \emptyset$ can  be proved in the same way. If $A_1\neq \emptyset$, then  there exists  a maximal interval $I_1:= (\tau_1,\tau_2) \subseteq A_1 \subseteq I$. Since $I_1$ is maximal, we clearly have $j_\e(\tau_i) = k_\e(\tau_i)$ for $i=1,2$. By construction, and in view also of ODE uniqueness, we have 
\begin{equation}\label{ordinesf}
j_\e'(\tau_1)>   k_\e'(\tau_1) > 0, \qquad 
j_\e'(\tau_2) <   k_\e'(\tau_2).
\end{equation}
Set $\tau_{min}=\tau_2$ if $j_\e'>0$ in $(\tau_1,\tau_2)$, and otherwise we set $\tau_{min}$ to be the minimal $\tau\in (\tau_1,\tau_2)$ such that  $j_\e'(\tau) = 0$. Both in the case $\tau_{min}< \tau_2$ and $\tau_{min}=\tau_2$, in view of \eqref{ordinesf} we deduce that the trajectories corresponding to $j_\e$ and $k_\e$ cross each other in the phase space $(h,h')\in\R^2$, i.e., there exists $t_1, \, t_2 \in (\tau_1, \tau_{min})$ with 
\begin{equation}\label{cross}
(j_\e(t_1),j_\e'(t_1) ) = (k_\e(t_2),k_\e'(t_2) ),
\end{equation}
and   we may assume that 
$t_1$ and $t_2$ are the minimal times such that \eqref{cross} holds. By construction $j_\e$ and $k_\e$ are strictly increasing in $(\tau_1,\tau_{min})$, so that we can consider their inverse, and we have  $j_\e'(j_\e^{-1}(h))>k_\e'( k_\e^{-1}(h))$ for all $h\in (j_\e(\tau_1), j_\e(t_1)) = (k_\e(\tau_1), k_\e(\tau_2))$. Since 
$$
t_1 - \tau_1 = \int_{j_\e(\tau_1)}^{j_\e(t_1)} \frac{d h}{j_\e'(j_\e^{-1}(h))},
\qquad
t_2 - \tau_1 = \int_{k_\e(\tau_1)}^{k_\e(t_2)} \frac{d h}{k_\e'(k_\e^{-1}(h))},
$$ 
we deduce that $t_1 < t_2$.
In addition,   by construction we have 
$$
\frac{j_\e''(t_1)}{j_\e'(t_1)}  = \frac{d}{dh} j_\e'(j_\e^{-1}(h))|_{j_\e(t_1)} \le  \frac{d}{dh} k_\e'(k_\e^{-1}(h))|_{k_\e(t_2)} = \frac{k_\e''(t_2)}{k_\e'(t_2)} = \frac{k_\e''(t_2)}{j_\e'(t_1)}, 
$$
and hence $j_\e''(t_1)\le k_\e''(t_2).$ On the other hand,
equation \eqref{eq1dim} implies that
$$
j_\e''(t_1) = - (n-1)\tanh t_1 \, j_\e'(t_1) - f_\e(j_\e(t_1)) > -(n-1) \tanh t_2 \, k_\e'(t_2) - f_\e(k_\e(t_2))=
k_\e''(t_2),
$$
that together with  $j_\e''(t_1)\le k_\e''(t_2)$ provides a contradiction. This shows that $j_\e = k_\e$, and hence concludes the proof that $k_\e$ is the only energy minimizer in $I$ with respect to its own boundary values. 
\par
Now we show that $k_\e$ is an energy minimizer among all smooth functions $h$ such that $h({\pm\infty}) = \pm 1$. To this purpose, let $\f_m(\tau):= \f(\tau/m)$ be a sequence  of standard smooth cut-off functions, i.e., $\varphi \in C^\infty_0 (\mathbb{R})$, $0\le \f \le 1$, $\f \equiv 0$ for $|\tau|\ge 1$, $\f\equiv 1$ for $|\tau|\le 1/2$. Given any smooth $h_\e$ with finite energy such that $h_\e^{\pm\infty} = \pm 1$, 
we set $h_{\e,m} : = \f_m k_\e + (1-\f_m) h_\e$. Since $k_\e$ minimizes the energy in any interval $I_m:=(-m,m)$, we have 
\begin{equation}\label{2contr}
E_\e(k_\e,I_m) \le E_\e(h_{\e,m}, I_m) \le E_\e(h_{\e,m}, I_m\setminus I_{m/2})+  E_\e(h_{\e}),
\end{equation}
where $E_\e(h,J)$ denotes the integral on the set $J$ of the energy density of $h$ defined in  \eqref{1dene}. It is easy to check that, as $m\to\infty$,  $E_\e(h_{\e,m},I_m\setminus I_{m/2}) \to 0$ and, in view of  \eqref{2contr},  we easily conclude  $E_\e(k_\e)\le  E_\e(h_{\e})$, i.e., $k_\e$ is a minimizer. 
\par
Finally, we pass to the proof of the uniqueness of $h_\e$ as given by Proposition \ref{odsprop}. To this purpose let $k_\e$ be a solution  to \eqref{eq1dim} vanishing at zero, and let us prove that $h_\e = k_\e$.  Notice that, in view of the previous part, both $h_\e$ and $k_\e$ are energy minimizer, hence, by standard odd reflection arguments, we have
$$
E_\e(h_\e,\R^-)=E_\e(h_\e,\R_+)=E_\e(k_\e,\R^-)=E_\e(h_\e,\R_+).
$$
Therefore  also the function $j_\e$ defined as $h_\e$ in $\R^-$ and as $k_\e$ on $\R_+$ is an energy minimizer. Thus, $j_\e$ satisfies \eqref{eq1dim} and by standard ODE regularity we deduce that $h_\e$ and $k_\e$ have same derivative at the origin, therefore they coincide by ODE uniqueness. 
\end{proof}

\subsection{Existence and uniqueness of elementary solutions of the PDE}\label{odsthmss} 
We are in a position to prove Theorem \ref{odsthm}. 
It is clear by our construction that, for $\Sigma_0 = \{x_1=0\}$ in the half-space model,  the function $U_\e(x):= h_\e(\tilde d(x,\Sigma_0))$ is a solution to equation \eqref{equa} with boundary conditions as $x_n\searrow 0$ given by $U_\e(x) = \text{ sgn}(x_1)$. Clearly, such solution can be viewed in the Poincar\'e ball model, and the corresponding boundary conditions are given by  $U_\e(x) = \pm 1 $ on two disjoint half spheres of the sphere  at infinity  $\sinf$.
\par
In the general case, we set $U_\e(x)=h_\e (\tilde d (x, \Sigma))$ and we may assume $\Sigma = T(\Sigma_0)$ for some hyperbolic isometry $T$. 
By definition of $U_\e$ we have
$$
U_\e(T(x))= h_\e( \tilde d(T(x),\Sigma)) =  h_\e( \tilde d(T(x),T (\Sigma_0))) =h_\e( \tilde d(x, \Sigma_0)).
$$ 
Since equation \eqref{equa} is invariant under isometries, we conclude that  $U_\e(x)$ is a solution of equation \eqref{equa}, and by construction it clearly satisfies the desired boundary conditions.  
Moreover,  as a direct consequence of Proposition \ref{odsprop} and of the previous equalities, we also deduce that $U_\e\equiv 0$ on $\Sigma$, and $U_\e(x) \to \text{ sgn}( \tilde d(x,\Sigma))$  locally uniformly in $B_1\setminus \Sigma$ as $\e\to 0$. 
\par
Let us pass to the proof of the minimality property of $U_\e$.  We will work in the half space model, and since the local minimality property  also is invariant under isometries, we may assume without loss of generality  $\Sigma = \Sigma_0=\{x_1=0\}$.  Our proof is based on a uniqueness argument and a sliding technique, inspired by the two dimensional analysis done in \cite{CGS} 
\par
To prove the minimality of $U_\e$ in any regular open set $A\subset\subset \R_+^n$, we will  prove  indeed that   any solution $\tilde U_\e$ of equation  \eqref{equa} in $A$, with $-1\le \tilde U_\e \le 1$ and coinciding with $U_\e$ on $\partial A$,  is in fact equal to $U_\e$. This is  enough to conclude since in an open set $A$ any minimizer  is clearly a solution.
\par
Note that, by standard truncation arguments,  any local minimizer $\tilde U_\e$ satisfies $|U_\e|\le 1$, and indeed $|\tilde U_\e|<1$ in $\overline A$ by standard maximum principles.  Therefore, it remains  to  prove that $\tilde U_\e  = U_\e$. We will prove the inequality $\tilde U_\e  \le U_\e$, the other inequality being analogous. 
\par
Since $h_\e$ is increasing, we deduce by construction that also $U_\e$ is increasing with respect to $x_1$. As a consequence, we have that the functions $U_{\e,\tau}(\, \cdot \,):= U_{\e}(\, \cdot \, + \tau e_1)$ are well ordered, i.e., $\tau_1 < \tau_2 \Rightarrow U_{\e,\tau_1} < U_{\e,\tau_2}$, and $U_{\e,\tau}\to \pm 1$ uniformly in $\overline A$ as $\tau \to \pm \infty$. Since $-1<\tilde U_\e < 1$ in $\overline A$, we have that $\tilde U_\e < U_{\e,\tau}$ for $\tau$ large enough. By continuity there exists  a minimum $\tau\in \R$, denoted by  $\tau_{min}$,   satisfying $\tilde U_\e \le U_{\e,\tau}$ in $\overline A$. Clearly we have $\tau_{min}\ge 0$ because of the values at the boundary. If $\tau_{min}>0$,  then there exists $x\in  A$ with $\tilde U_\e(x) = U_{\e,\tau_{min}}(x)$, but since $U_{\e,\tau_{min}}$ is also a solutions, this is in contradiction with  standard maximum principles (see for instance \cite[Lemma 2.3]{CGS}).  Therefore we have $\tau_{min}=0$, and hence  $\tilde U_\e \le U_{\e}$. Arguing similarly we also get $U_\e\ge \tilde{U}_\e$, whence $U_\e\equiv \tilde{U}_\e$ and  $U_\e$ is a local minimizer. 
\par
Finally, let $u_\e$ be a solution to \eqref{equa} satisfying the same boundary conditions of  $U_\e$ on $\sinf$. According to \cite[Theorem 3.5]{BM}, $u_\e$ is indeed one-dimensional, i.e., $u_\e(x)= k_\e(\tilde d(x,\Sigma))$ for a suitable $k_\e$ solving problem \ref{eq1dim}. Since $u_\e \equiv 0 $ on $\Sigma$, we have $k_\e(0)=0$. By Proposition \ref{nac} we infer $h_\e = k_\e$, so that $u_\e = U_\e$, that concludes the proof of Theorem \ref{odsthm}.

\begin{oss}
It would be interesting to know whether the uniqueness statement in Theorem \ref{odsthm} still holds without the assumption   $U_\e \equiv 0 $ on $\Sigma$. 
In light of \cite[Theorem 3.5]{BM}, this uniqueness property  is indeed equivalent to the uniqueness of the solution $h$ of \eqref{eq1dim}. We have proved this uniqueness property in Proposition \ref{nac} only under the  additional assumption $h(0)=0$ which corresponds to $U_\e \equiv 0$ on $\Sigma$. 
Finally, we notice that the uniqueness property is know to fail in the Euclidean context because of the translation invariance of the equation. On the other hand,   in the hyperbolic space, due to the presence of the weight $\cosh^{n-1} \tau$ in the energy functional \eqref{1dene}  there is no translation invariance, and since the weight is increasing in $|\tau|$,  it seems very likely that the minimizer vanishes at zero. 
\end{oss}

\section{Multidimensional phase transitions}\label{mpt}
In this section we will construct our minimizing phase transitions in  hyperbolic space with prescribed boundary value at infinity. More precisely,  we will construct global solutions for the equation \eqref{equa} that are local minimizer of the energy functional \eqref{enfun}, and satisfying the prescribed boundary conditions $u=\pm 1$ on given open subsets $\Om^+$, $\Om^-\subset \sinf$.
\par
As a building block, we will use the one dimensional  solutions obtained in Section \ref{odsSec} to construct barriers $\underline\psi_\e$, $\overline \psi_\e$, defined  as  the supremum and the infimum, respectively,  of  suitable one dimensional solutions. Such barriers, in view of the inequality $\underline \psi_\e \le u_\e \le \overline \psi_\e$,   will be used to control the behaviour of the solution $u_\e$ at infinity. We adopt a strategy similar to the one suggested in \cite{GJ} and used in \cite{JMZH} in  dimension two. Thus, we construct $u_\e$ as the limit of energy minimizers $u_{\e,R}$ defined on a family of exhausting subdomains, and with free boundary value between $\underline \psi_\e$ and $\overline \psi_\e$. In view of comparison principles, we show that the inequality $\underline \psi_\e\le u_{\e,R}\le \overline \psi_\e$ holds also in the interior of each subdomain, and it yields in the limit $\underline \psi_\e\le u_{\e}\le \overline \psi_\e$ in the whole $\Hyp^n$, ensuring in this way that 
$u_\e$ attains the desired boundary values at infinity.
\par 
As in Theorem \ref{odsthm}, let $C^+$ and $C^-$ be disjoint open spherical caps in $\sinf$ (in the Poincar\'e ball model) with common boundary $L$. The sets $C^+$ and $C^-$ can be equivalently described as  $C^+= I_{r^+}(p^+)$, $C^-= I_{r^-}(p^-)$, for suitable antipodal point $p^{\pm}\in \sinf$ and suitable radii $r^{\pm}$ with $r^+ + r^- = \pi$,  where $I_r(p)$ denotes the ball of radius $r$ and center $p\in \sinf$ with respect to the standard Riemannian distance on the sphere. Moreover, whenever $r^+\neq r^-\neq \pi/2$, the sets $C^+$ and $C^-$ uniquely  determine (and at the same time they are determined by) a unique Euclidean ball $B$ (actually a half space in the limiting case $r^+= r^- =\pi/2$). Indeed, let  $C_{min}$ be the smallest spherical cap between $C^+$ and $C^-$. Then  there exists a unique Euclidean ball $B$ such that
 $B\cap \sinf = C_{min}$  and $\Sigma:= \partial B\cap B_1$ touches $\partial B_1$ orthogonally along $L=\partial \Sigma$. 
\par
In our construction of the solution we will use the signed distance function $\tilde d(x,\Sigma)$ from the set $\Sigma$ defined above,  with the convention  $\tilde d(x,\Sigma) \to \pm \infty$ as $x\to C^{\pm}$ respectively.
\par
Note that when $r^+ = r^- = \pi/2$, the corresponding $\Sigma$ is a $(n-1)$-dimensional disk; e.g.,
$\Sigma := \Sigma_0 = \{x_n= 0\}$. Moreover, all the sets $\Sigma$'s are isometrically equivalent, and hence in particular there they are isometrically equivalent to $\Sigma_0$.

\subsection{Sub-solutions and super-solutions}
In this part we will define  suitable barriers for the solution to equation \eqref{equa} which we will construct in the next paragraph. The idea here is to combine one dimensional solutions provided in Theorem \ref{odsthm}, corresponding to two families of  spherical caps, exhausting   the open sets $\Om^+$ and $\Om^-$ respectively. 
\par The following lemma establishes a monotonicity property for the family of one dimensional solutions.
\begin{lemma}\label{sior}
Let $\{C_1^+,C_1^-\}$ and $\{C_2^+,C_2^-\}$ be two pairs of spherical caps in $\sinf$, and let $U_\e^1$ and $U_\e^2$ be the corresponding one dimensional solutions given by Theorem \ref{odsthm}.
\par
Then we have $U_\e^1\le U_\e^2$ if and only if $ C_1^+\subseteq C_2^+$. Moreover we have strict inequality $U_\e^1(x) < U_\e^2(x)$ for every $x\in \Hyp^n$   whenever  the inclusion    
$C_1^+\subset C_2^+$ is strict.
\end{lemma}
\begin{proof}
Assume $U_\e^1\le U_\e^2$ and  let $p\in    C_1^+$. Since $U_\e^2\le 1$, we have  $ 1\leq \lim_{x\to p} U_\e^1 \le \lim_{x\to p} U_\e^2 \leq 1$, i.e., $x\in C^+_2$.   We conclude that $C_1^+\subseteq C_2^+$. 
\par
Now, Let $\Sigma_1$ and $\Sigma_2$ be the zero level sets of $U_1$ and $U_2$ respectively, so that $U_\e^1 (x) = h_\e(\tilde d (x,\Sigma_1))$, $U_\e^2 (x) = h_\e(\tilde d (x,\Sigma_2))$.  

If $C^+_1\subseteq C^+_2$, then $d(\cdot,\Sigma_1) \le d(\cdot,\Sigma_2)$, 
and the inequality is strict whenever the inclusion is strict. Since $h_\e$ is strictly increasing the conclusion follows.  
\end{proof}

Another useful property of one dimensional solutions is that they are essentially closed under uniform 
convergence of compact sets. More precisely we have the following lemma 

\begin{lemma}\label{odr}
Let $U_\e^m$ be one dimensional solutions, corresponding to pairs of spherical caps $C_m^+$, 
$C_m^-$ with common boundary.
Up to a subsequence, we have $U_\e^m\to U_\e$ locally uniformly, for  some  solution  $U_\e$ of equation 
\eqref{equa}. Moreover,   either $U_\e \equiv \pm 1$ or it is a one dimensional solution corresponding to some spherical caps $C^+$, $C^-$.
\end{lemma}

\begin{proof}
Since the compactness property of $U_\e^m$ is clearly invariant by composing  $U^m_\e$ with a convergent sequence of  isometries, we may assume without loss of generality  
that  $C_m^+$ and $C_m^-$ are concentric. 
Let $\Sigma_\e^m$ be the spherical caps corresponding to the  zero level sets of  $U_\e^m$.  
Let  $T_m$ be the hyperbolic isometries mapping $\Sigma_m$ into $\Sigma_0 := \{x_n=0\}$, corresponding to pure dilations in the half-space model. Up to a subsequence, we have that either $T_m$ converge locally uniformly to some limit isometry $T$, or  $T_m$ converges locally uniformly to the constant map $T(x) \equiv p$, where $p$ is a center of the  concentric caps $C_m^+$, $C_m^-$.  
\par
In the first case, set $\Sigma:= T^{-1}(\Sigma_0)$, 
 $$U_\e(x):= h_\e(\tilde d(x, \Sigma))= h_\e(\tilde d(T x, T \Sigma))=
 h_\e(\tilde d(T x, \Sigma_0)),
 $$ 
 and $C^+$ and $C^-$ the corresponding spherical caps. Since  $h_\e$ is continuous and $T_m$ converges to $T$ locally uniformly and 
$$
U_\e^m(x)= h_\e(\tilde d(x, \Sigma_m)) = h_\e(\tilde d(x, T_m^{-1}\Sigma_0)) = 
h_\e(\tilde d(T_m x, \Sigma_0)),
$$ 
 we deduce that $U_\e^m$ converges locally uniformly to the function $U_\e$, that has all the desired properties.   
\par
Finally, in the second case we  have $\Sigma_m\cap K=\emptyset$ for every compact set  $K\subset \Hyp^n$ and $m$ large enough. Then, it is easy to see that $\tilde d(x,\Sigma_m) \to \pm \infty$ locally uniformly in $\Hyp^n$, hence  $U_\e^m \to \pm 1$ locally uniformly. 
\end{proof}

Now we will construct the barriers $\underline \psi_\e$, $\overline \psi_\e$. Let $\Om^+$, $\Om^-$  
be disjoint open subset of $\sinf$, and let $(C^+, C^-)$ denote any pair of  disjoint spherical caps in $\sinf$ with common boundary. We set
\begin{equation}\label{fpfm}
\F^+:= \{ (C^+, C^-): \, C^+ \subset \Om^+\},\qquad \F^-:= \{ (C^+, C^-): \, C^- \subset \Om^-\}.
 \end{equation}
Given a pair $(C^+,C^-)$, the corresponding one dimensional solution provided by Theorem \ref{odsthm} will be denoted by 
$U_\e^{C^+,C^-}$. Finally, for every $x\in \Hyp^n$ we set 
\begin{equation}\label{barriers}
\underline \psi_\e(x): = \sup \{U_\e^{C^+,C^-}(x), \, (C^+,C^-) \in \F^+ \}, \quad  \overline \psi_\e(x)
: = \inf \{U_\e^{C^+,C^-}(x), \, (C^+,C^-) \in \F^- \}.
\end{equation}
In the next proposition we summarize some properties satisfied by the barriers just introduced.

\begin{proposition}\label{psitocca}
The barriers $\underline \psi_\e$, $\overline \psi_\e$ defined in \eqref{barriers}  are  Lipschitz in $\Hyp^n$ with respect to the hyperbolic metric. In addition, we have
$-1< \underline \psi_\e(x) \le \overline \psi_\e(x) <1$ for all $x\in \Hyp^n$. Moreover  
for every $p^{\pm} \in \Om^{\pm}$ we have $\lim_{x\to p^+} \underline\psi_\e(x) = 1$ and 
$\lim_{x\to p^-} \overline\psi_\e(x) = -1$.
\par
Finally, either 
$\underline \psi_\e(x) < \overline \psi_\e(x) $ for every $x\in \Hyp^n$, or $\underline \psi_\e \equiv  \overline \psi_\e = U_\e$ for some one-dimensional solution $U_\e$,  and $\Om^+$  and $\Om^-$ are open disjoint spherical caps with common boundary.   
\end{proposition}

\begin{proof}
Given $\Sigma\subset \Hyp^n$, we clearly have that the signed distance $\tilde d(x,\Sigma)$ is $1$-Lipschitz with respect to the hyperbolic distance. Therefore, since the function $h_\e$ is Lipschitz, we deduce that the one dimensional solutions $U_\e(x) = h_\e(\tilde d(x,\Sigma))$ are Lipschitz in the hyperbolic space, with Lipschitz constant independent of $\Sigma$.
Therefore, passing to the supremum and the infimum respectively, we deduce that the same property is inherited by   $\underline \psi_\e$ and $ \overline \psi_\e$. 
\par
Note that given two pairs ($C_1^+,C_1^-)\in \F^+$, ($C_2^+,C_2^-)\in \F^-$, we always have
$C_1^+\subseteq C_2^+$. Thus the corresponding one-dimensional solutions $U_\e^1$ and $U_\e^2$ satisfy $-1< U_\e^1\le U_\e^2<1$, with strict inequality unless  (see Lemma \ref{sior}) ($C_1^+,C_1^-)= (C_2^+,C_2^-)$. Taking the supremum and the infimum respectively on $\F^+$ and $\F^-$ we obtain
$-1< \underline \psi_\e(x) \le \overline \psi_\e(x) <1$.
\par
Since the families  of spherical caps in $\F^{\pm}$ give  coverings of $\Om^\pm$, then  the limits $\lim_{x\to p^+} \underline\psi_\e(x) = 1$ and 
$\lim_{x\to p^-} \overline\psi_\e(x) = -1$ follow easily by construction of $\underline \psi_\e, \,\overline \psi_\e$, and in particular by the asymptotic behavior of one-dimensional solutions, stated in Theorem \ref{odsthm}. 
\par
Finally, assume that equality holds at some point $x\in \Hyp^n$. Then there exists sequences 
$\{(C^+_{1,m}, C^-_{1,m})\} \subset \F^+$, $\{(C^+_{2,m}, C^-_{2,m})\} \subset \F^-$, $C^+_{1,m}\subseteq C^+_{1,m}$, such that the corresponding solutions $U^1_{\e,m}$, $U^2_{\e,m}$ satisfy 
$U^1_{\e,m} \le U^2_{\e,m}$ and $\lim_{m\to\infty} U^1_{\e,m}(x) = \lim_{m\to\infty} U^2_{\e,m} =l\in (-1,1)$.
Since $l\in (-1,1)$, by Lemma \ref{odr} we deduce that (up to a subsequence) $U^i_{\e,m} \to U^i_\e$ locally uniformly, for some one dimensional solutions $U^i_\e$, corresponding to some spherical caps $C^+_i$, for $i=1,\,2$. Clearly $U^1_\e \le U^2_\e$ in $\Hyp^n$, hence $C^+_1\subseteq C^+_2$ by Lemma \ref{sior}. On the other hand,  $U^1_\e(x) = U^2_\e(x) =l $, so that, in view of  Lemma \ref{sior}, $U^1_\e \equiv \underline\psi_\e \equiv\overline\psi_\e \equiv U^2_\e$ and  the proof is complete.  
\end{proof}

\subsection{Construction of solutions of the PDE}
Here  we will prove Theorem \ref{mainthm}. In particular, we will construct  global solutions for the equation \eqref{equa} that are local minimizer of the  energy functional \eqref{enfun}, and satisfying the prescribed boundary conditions $u=\pm 1$ on given open subsets $\Om^+$, $\Om^-\subset \sinf$.
\par
First we consider an increasing sequence $r_k\to 1$, and construct   energy minimizers $u_{\e,k}$ defined on $B_{r_k}$,  with  $\underline \psi_\e \le u_{\e,k} \le \overline \psi_\e$. Then, letting  $k \to \infty$, we obtain by compactness   a limit  solution $u_\e:= \lim_k u_{\e,k} $ such that $\underline \psi_\e \le u_\e \le \overline \psi_\e$, in the whole hyperbolic space $\Hyp^n$.
\par
Let $r_k\in (0,1)$ be fixed. The solution $u_{\e,k} $ in $B_{r_k}$ is defined as a minimum point of the following minimization problem,
\begin{equation}\label{minprork}
\min \{  \E_\e(u,B_{r_k}), u\in H_{\Psi_\e}^1(B_{r_k}) \}
\end{equation}
where $\E_\e$ is defined in \eqref{enfun3} and $H_{\Psi_\e}^1(B_{r_k})$ denotes the set of $H^1$ functions with traces on $\partial B_{r_k}$ between $\underline\psi_\e$ and $\overline\psi_\e$. 
\begin{proposition}\label{proprk}
The minimum problem \eqref{minprork} admits a minimizer $u_{\e,k}\in H_{\Psi_\e}^1(B_{r_k})$, satisfying $\underline \psi_\e \le u_{\e,k} \le \overline \psi_\e$ in $B_{r_k}$. 
Moreover $u_{\e,k} \in C^{2}(B_{r_k})$ is a classical solution of \eqref{equa}.
\end{proposition}

\begin{proof}
Assume first $\underline \psi_\e(x) = \overline \psi_\e(x)$ for some $x\in \Hyp^n$.
In view of Proposition \ref{psitocca} we  have $\underline \psi_\e = \overline \psi_\e = U_{\e}$ for   some one dimensional solution $U_{\e}$. Since $U_{\e}$ is a minimizer of \eqref{minprork}  with respect to his own boundary conditions (indeed, following the proof of Theorem \ref{odsthm}, the unique minimizer), we conclude that  $u_{\e,k}^m \equiv U_{\e}$ is a solution of \eqref{minprork}. 
\par
Now assume $\underline \psi_\e < \overline \psi_\e$ in $\overline B_1$. 
Let $u_{\e,k}^m$ be a minimizing sequence for problem \eqref{minprork}. Because of the strict inequality,  by standard truncation and approximation arguments we may assume that  $u_{\e,k}^m$ are smooth up to the boundary and less than or equal to one in modulus. Notice that the energy functional $\E_\e$ in  \eqref{minprork} is sequentially weakly lower semi-continuous on $H^1$, so that, following the direct method of calculus of variations,  the energy $\E_\e$ admits a minimizer among all  $u\in H^1(B_{r_k})$ with $u=g$ on $\partial B_{r_k}$, where $g\in H^{1/2}(\partial B_{r_k})$ is a given boundary condition.   
Therefore, by further minimization we may assume that $u_{\e,k}^m$ minimize \eqref{minprork} with respect to their own boundary conditions $g_{\e,k}^m:=Tr(u_{\e,k}^m)$, where $Tr$ denotes the trace operator. In particular $u_{\e,k}^m$  solve equation \eqref{equa3} with smooth boundary conditions $g_{\e,k}^m$, so that they are smooth (say $C^2$) up to the boundary by standard regularity theory for elliptic equations. 
\par
Now we aim to prove the inequality $\underline \psi_\e \le u^m_{\e,k} \le \overline \psi_\e$ in $\overline B_{r_k}$. We will prove just the inequality $\underline \psi_\e \le u^m_{\e,k}$, the proof of the other one being entirely similar. By definition of $\underline \psi_\e$, it is enough to prove the inequality
$U_\e \le u^m_{\e,k}$ in $\overline  B_{r_k}$ for every one-dimensional solution $U_\e$ corresponding to some  spherical caps  $(C^+,C^-)\in \F^+$, according with \eqref{fpfm}. Since $U_\e \le u^m_{\e,k}$ holds on $\partial B_{r_k}$ and since the solutions $u^m_{\e,k}$ belong to
$C^2(B_{r_k}) \cap C^0(\overline B_{r_k})$, we can repeat the sliding argument used in the proof of Theorem \ref{odsthm} to obtain the inequality $U_\e\le u^m_{\e,k}$ in $\overline{B_{r_k}}$.  This concludes the proof of 
\begin{equation}\label{aimata}
\underline \psi_\e \le u^m_{\e,k} \le \overline \psi_\e  \qquad \text{ in } \overline B_{r_k}. 
\end{equation}
Now, letting $m\to \infty$, up to a subsequence we have $u^m_{\e,k}\weak  u_{\e,k}$ for some $u_{\e,k}\in H^1_{\Psi}(B_{r_k})$. Since  $u^m_{\e,k}$ is a minimizing sequence, by lower semi-continuity we conclude that $u_{\e,k}$ is a minimum for the variational problem \eqref{minprork}. 
Clearly, $u_{\e,k}$ is a solution of the corresponding Euler-Lagrange equation \eqref{equa3}, and hence it is $C^2$ in $B_{r_k}$ by standard regularity theory. 
Finally, as $m\to \infty$, by \eqref{aimata} we deduce  $\underline \psi_\e \le u_{\e,k} \le  \overline \psi_\e$ in $B_{r_k}$ as desired.
\end{proof}
We are in a position to complete the proof of Theorem \ref{mainthm}. As already explained, the last step consists in taking the limit of the solutions $u_{\e,k}$ given by Proposition \ref{proprk}, as $r_k \to 1$. 

\begin{proof}[Proof of Theorem \ref{mainthm}]
Let $u_{\e,k}$ be the solutions in $B_{r_k}$ given by Proposition \ref{proprk}. Since they are  equi-bounded and smooth, by standard elliptic regularity theory they are  bounded in $C^{2,\alpha}_{loc}$, and hence they are precompact in $C^2_{loc}(B_1)$. Hence, up to a subsequence we may assume that $u_{\e,k}$ converge to some $u_\e$  in $C^2_{loc}(B_1)$. Clearly $u_\e$, being limit of locally minimizing solutions, is itself a local minimizer of the energy $\E_\e$ in $\eqref{enfun}$, and $u_\e\in C^2(\Hyp^n)$ is a classical solution of \eqref{equa}. Finally, since $\underline \psi_\e \le u_{\e,k} \le  \overline \psi_\e$ in $B_{r_k}$, letting $r_k\to 1$
we get $\underline \psi_\e \le u_{\e} \le  \overline \psi_\e$ in $\Hyp^n$.  In view of Proposition \ref{psitocca} we conclude that $u_\e$ satisfies the desired boundary conditions, i.e. $u_\e\in C^0(\Hyp\cup \Om^+ \cup \Om^-)$, 
and $u_\e(x) = \pm 1$ on $\Om^{\pm}$.
\par
Next, we prove the inclusion $\Sigma_\e\subset 
\overline{conv(F)}$.  
First we recall that by closed half-spaces we mean the closure in $\Hyp^n$ of any connected component of  $\Hyp^n\setminus \Sigma$, where  $\Sigma$ is  (in the ball model of $\Hyp^n$) a spherical cap  touching $\sinf$ orthogonally.  Notice that by Theorem \ref{odsthm}, we can identify the half spaces with the sets of positivity of elementary solutions, and such correspondence is bijective.  Now we  consider the family of all closed  half spaces such that their Euclidean closure in $\overline B_1$ contains $F$. Then it is a standard fact that  $\overline{conv(F)}\cap B_1$ coincides with the intersection of  all such closed  half spaces.
\par 
 By \eqref{barriers} it follows that if $U_\e=U_\e^{C^+,C^-}$ with  
$(C^+,C^-) \in \F^+$, then the set of negativity of $u_\e$ is contained in the set of negativity of 
$U_\e$, and an analogous inclusion relation holds for the set of positivity of $u_\e$ and for any 
elementary solution $U_\e=U_\e^{C^+,C^-}$, with $(C^+,C^-) \in \F^-$. Thus, varying $(C^+,C^-) $ in $\F^+$ and $\F^-$ respectively, we deduce 
$$
\{u_\e \le 0 \} \subset \bigcap_{(C^+,C^-)\in \F^+} \{U_\e^{C^+,C^-} \le 0\},\qquad
\{u_\e \ge 0 \} \subset \bigcap_{(C^+,C^-)\in \F^-} \{U_\e^{C^+,C^-} \ge 0\},
$$ 
and therefore
\begin{equation}\label{nondalabel}
\{u_\e = 0 \} \subset \bigcap_{(C^+,C^-)\in \F^+} \{U_\e^{C^+,C^-} \le 0\} 
\bigcap
\bigcap_{(C^+,C^-)\in \F^-} \{U_\e^{C^+,C^-} \ge 0\}.
\end{equation}
Since to each closed  half space with Euclidean closure in $\overline B_1$ containing $F$
corresponds  (either the positivity or the negativity set of) an elementary solutions  $U_\e=U_\e^{C^+,C^-}$ (with  $(C^+,C^-)$ either in  $\F^-$ or in $\F^+$, respectively)  the inclusion \eqref{nondalabel} is equivalent to  $\Sigma_\e\subset \overline{conv(F)}$ .
\par
The inclusion $\overline{\Sigma_\e} \cap \sinf \subset F$ is now a direct consequence of $\Sigma_\e\subset \overline{conv(F)}$ and $\overline{conv(F)} \cap \sinf = F$, so we pass 
to prove   the inclusion $\partial \Om^+ \cap \partial \Om^- \subseteq \overline{\Sigma_\e} \cap \sinf$. To this purpose, it is enough to notice that if $x\in \partial \Om^+ \cap \partial \Om^-$, 
then for every positive $\rho$ there exists a continuous path $\gamma\subset B_\rho(x)$ joining a point of $\Om^+$ with a point of $\Om^-$. By continuity of $u_\e$, we deduce that there exists $\overline x\in  B_\rho(x)$ with $u_\e(\bar x) = 0$, i.e., $\bar x\in \Sigma_\e$. By the arbitrariness of $\rho$, we deduce  $x\in\overline{\Sigma_\e} \cap \sinf$ for all $x\in \partial \Om^+ \cap \partial \Om^-$, that is   $\partial \Om^+ \cap \partial \Om^- \subseteq \overline{\Sigma_\e} \cap \sinf$. 
\par 
Finally we prove that $\Sigma_\e$ is always a $C^2$ hypersurface for $n\leq 7$ and $\e\leq \e_0$ sufficiently small. We work in the Poincar\'e ball model and we prove the claim arguing by contradiction. Indeed, if by contradiction $\Sigma_\e$ are not smooth, then by the implicit function theorem we  have that, for some sequence $\epsm \to 0$, there are locally energy minimizing solutions $\{ u_\epsm \}$ of \eqref{equa3} such that (up to hyperbolic isometries) $u_\epsm(0)=0$ and $\nabla u_\epsm(0)=0$.
We introduce the scaled functions $\tilde u_m \in C^2(2{\epsm}^{-1}B_1)$ as $\tilde u_m(x)=u_\epsm(\frac{\epsm}{2} x)$, so that for each $m \geq1$ each $\tilde u_m$ solves
\begin{equation}
\label{scaledeq}
(1-|\epsm x|^2)^2 \Delta  \tilde u_m + (n-2) ( 1-|\epsm x|^2 ) \epsm x \cdot \nabla \tilde u_m  +f(u_\epsm)=0.
\end{equation}
According to the standard elliptic regularity theory for \eqref{scaledeq} the sequence $\{ \tilde u_m\}$ is compact in $C^2_{\rm{loc}}$, so, up to subsequences, there exists $\tilde u \in C^2 (\mathbb{R}^n)$ such that as $m \to \infty$ we have  $\tilde u_m \to \tilde u$ in $C^2_{\rm{loc}}$, $\tilde u$ is  an entire solution 
of $\Delta \tilde u + f(\tilde u)=0$ , $\tilde u(0)=0$ and $\nabla \tilde u (0)=0$. Since local energy minimality  passes to the limit under smooth convergence,  it's easy to check that the limiting function $\tilde u$ is also a local energy minimizer of the energy functional \eqref{enfun} on $\mathbb{R}^n$ with the standard metric. Since $\tilde u \not \equiv \pm 1$, according to \cite{SO}, Theorem 2.3,  for $n\leq 7$ we have $\tilde u(x)=g(a \cdot x)$ for some unit vector $a \in \mathbb{R}^n$ and some strictly increasing function $g \in C^2 (\mathbb{R})$ vanishing at the origin which solves the ODE $g^{\prime \prime}+f(g)=0$ on the real line. On the other hand, since $g^\prime (0)=\nabla \tilde u (0) \cdot a =0$ and $f$ is $C^1$ and odd we conclude $g\equiv 0$ by ODE uniqueness, which is a contradiction because $g$ is strictly increasing.
\end{proof}

\subsection{Asymptotic behavior and fine properties of solutions}
In this paragraph we study the asymptotic behaviour of the solution constructed in Theorem \ref{mainthm} under the assumption that $L:=\partial \Om^\pm $ is a $C^1$ hypersurface in the sphere at infinity $\sinf$. First, in Proposition \ref{blowup} 
we show that blowing up  the solution $u_\e$ 
around a point of $L$  the sets $\Om^\pm$ converge (under scaling) to a pair of $(n-1)$-dimensional half spaces, while 
$u_\e$ converges to the corresponding one dimensional solution given by Theorem \ref{odsthm}.
As a consequence, we will be in a position to prove Theorem \ref{ab}, showing that the zero level set of $\Sigma$ touches orthogonally   the sphere at infinity along $L$, and proving the asymptotic expansion \eqref{abe} for $u_\e$ near $L$. In the following we set $\nu_L(p)$ the  inner unit normal to $\partial \Om^+$ at $p$.  Finally,   we define $g_\e$ as  the solution to problem \eqref{eq0dim} vanishing at zero, corresponding to the unique solution $h_\e$ to \eqref{eq1dim} vanishing at zero (see Proposition \ref{odsprop} and Proposition \ref{nac}). 

\begin{proposition}\label{blowup}
Let $\Om^+$ and $\Om^-$ be disjoint open subsets of $\sinf$ with common boundary $L$, and assume that $L\subset \sinf$  is a smooth hypersurface of class $C^1$. Let moreover $\{p_k\}\subset L$ converging to some $p\in L$  and $\lambda_k \searrow 0$ as $k\to\infty$. 
Finally, let $u_\e$ be a local minimizer of the energy $\E_\e$ in \eqref{enfun}, such that $\underline\psi_\e \le u_\e \le \overline \psi_\e$, where $\underline\psi_\e $ and $\overline \psi_\e$ are defined in \eqref{barriers}.
\par
Then, in the half space model $\Hyp^n\simeq  \R^{n-1}\times (0,\infty)$ we have 
\begin{equation}\label{blowupconv}
u_\e(p_k + \lambda_k R_k y) \to g_\e \left(\frac{\nu_L(p) \cdot y}{y_n}\right) \qquad \text{ as } k\to\infty
\end{equation}
 in $C^2_{loc}(\R^{n-1}\times (0,\infty))$, for a suitable  sequence $\{R_k\}\subset O(\R^{n-1})$ converging to the identity,  with $R_k \nu_L(p) = \nu_L(p_k)$. 
\par
Finally, for every $\{q_k\}\subset \Hyp^n$ with $q_k \to q_\infty \in 
 \{ \nu_L(p) \cdot y \neq 0, \, y_n=0 \} \subset \R^{n-1}\times \{0\}  $ we have 
$u_\e(p_k + \lambda_k R_k q_k) \to \text{ sgn } \nu_L(p) \cdot q_\infty$ as $k\to\infty$.
\end{proposition}

\begin{proof}
Up to a translation we can always assume $p=0\in \mathbb{R}^{n-1}\times \{ 0\}$. Moreover,   up to a rotation $R\in O(\R^{n-1})$ we can  assume that $\nu_L(p)=(1,0,\ldots , 0)$, so that locally around $p$  we have $L = \{(f(x_2, \ldots ,x_{n-1}), x_2, \ldots ,x_{n-1})\}$  for some $ C^1 $ function $f$ such that $f(0,\ldots ,0) = 0$, $\nabla f(0,\ldots 0) = 0$, and
 $$
 \Om^+ = \{(x_1, \ldots, x_{n-1}) : x_1 > f (x_2, \ldots ,x_{n-1})\},
\quad
\Om^- = \{(x_1, \ldots, x_{n-1}) : x_1 < f (x_2, \ldots ,x_{n-1})\}.
$$
Since $p_k \to p$ as $k\to \infty$ and $f$ is $C^1$ we can choose rotations $R_k\in  O(\R^{n-1})\subset O(\R^{n})$  with $R_k\to Id$ such that  $R_k (1,0,\ldots,0) $ is  the inner unit normal to $\partial \Om^+$ at $p_k$. 
Let us set
$$
v_{\e,k}(y):= u_\e(p_k + \lambda_k (R_k y) ),
$$
so that $v_{\e,k}$ are smooth solutions to equation \eqref{equa2} in $\mathbb{R}^{n-1}\times (0,\infty)$, 
and let us prove that
\begin{equation}\label{limkblow}
v_{\e,k}(y) \to g_\e \left(\frac{y_1}{y_n}\right) \qquad    \text{ locally uniformly as } k\to\infty.
\end{equation}
To this purpose, let $B_1,\, B_2$ two given balls in $\R^{n-1}$ with $B_1\subset \{y_1 >0\}$ and
$B_2\subset \{y_1<0\}$. Since $L$ is $C^1$ we clearly have that, for $k$ large enough, 
$$
B_1\subset \lambda_k^{-1} R_k^{-1}(\Om^+ - p_k),
\qquad
B_2\subset \lambda_k^{-1} R_k^{-1}(\Om^- - p_k),
$$
or, equivalently,
\begin{equation}\label{palladentro}
C_{1,k}^+:= p_k + \lambda_k R_k B_1 \subset \Om^+,
\qquad
C_{2,k}^-:=  p_k + \lambda_k R_k B_2 \subset \Om^-.
\end{equation}
Let us consider the elementary solution $U^{1,k}_{\e}$ corresponding to the spherical cap 
$C_{1,k}^+$ (and to its complementary  $C_{1,k}^-$ in the sphere at infinity), and analogously let  $U^{2,k}_{\e}$ be the elementary solution corresponding to
$C_{2,k}^-$ (and $C_{2,k}^+$). 
By \eqref{barriers} and the assumption on $u_\e$ we have
$$
U^{1,k}_{\e}(x) \le \underline \psi_\e(x) < u_\e(x) <  \overline \psi_\e(x) \le U^{2,k}_{\e}(x)
\qquad \text{for all }x  \text{ in } \Hyp^n. 
$$ 
Changing variables in the previous inequality, we get
\begin{equation}\label{clearly}
U^{1}_{\e}(y)   < v_{\e,k}(y) <   U^{2}_{\e}(y), \qquad \text{for all }y  \text{ in } \mathbb{R}^{n-1}\times (0,\infty) \, , 
\end{equation}
where $U^{1}_{\e}$ and $U^{2}_{\e}$ are the elementary solutions corresponding to $B_1$, $ B_2$ and their complements.
Since $v_{\e,k}(y)$ are uniformly bounded solutions of  \eqref{equa2}, by standard a priori estimates we have that, up to subsequences, $v_{\e,k}(y)$ converges  in $C^2_{loc}(\mathbb{R}^{n-1}\times (0,\infty))$  to some function  $v_{\e,\infty}\in C^2 (\mathbb{R}^{n-1}\times (0,\infty))$ which solves \eqref{equa2}. Clearly, inequality \eqref{clearly} yields
\begin{equation*}
U^{1}_{\e}(y)   < v_{\e,\infty}(y) <   U^{2}_{\e}(y) \, , \qquad \text{for all }y  \text{ in } \mathbb{R}^{n-1}\times (0,\infty). 
\end{equation*}
Since $B_1\subset \{y_1>0\}$ and $B_2\subset \{y_1<0\}$ can be chosen arbitrarily, taking the supremum and the infimum respectively in the previous inequality, in view also of Proposition \ref{psitocca} we deduce  that
$
v_{\e,\infty}(y) = g_\e \left(\frac{y_1}{y_n}\right),
$
i.e., $v_{\e,\infty}$ is the elementary solution corresponding to the half spaces $C^+ = \{y_1>0\}$ and  $C^- = \{y_1<0\}$.   
By the uniqueness of the limit we conclude that the whole sequence $v_{\e,k}(y)$ converges to  $g_\e(y_1/y_n)$ in $C^2_{loc}(\mathbb{R}^{n-1}\times (0,\infty))$, i.e. 
 \eqref{blowupconv} holds.
 \par
Finally, we can always assume that $q_\infty \in B_1\cup B_2$, so that the  last statement
of the proposition easily follows from \eqref{clearly}, choosing $y=q_k$ and letting $k\to\infty$.
 \end{proof}
We are in a position to prove Theorem \ref{ab}. 

\begin{proof}[Proof of Theorem \ref{ab}]
The existence 
of an entire solution $u_\e\in C^2(\Hyp^n)\cap C^0(\Hyp^n\cup\sinf\setminus L)$ to equation \eqref{equa} satisfying the prescribed boundary conditions, that is a local minimizer of the energy $\E_\e$ in \eqref{enfun}, and with $\Sigma_\e\subset 
\overline{conv(L)}$ is provided by Theorem \ref{mainthm}.
\par
Now we pass to the proof of the regularity property of $\Sigma_\e:= u^{-1}_\e(0)$ and its orthogonality to $\sinf$,  using a  blow-up argument based on Proposition \ref{blowup}. Let $\{P_k\}\subset \Hyp^n$ be a sequence of points converging to some limit $P_\infty \in L$, and denote by $p_k$ a projection of $P_k$ on $L$, i.e., a point in $L$ of minimal Euclidean distance from $P_k$  in the half space model (with origin in $P_\infty$), so that  $\lambda_k:=|p_k - P_k| = \text{dist}_E(P_k, L)$ is the Euclidean distance between $P_k$ and $L$ and $p_k \to P_\infty = 0$ as $k \to \infty$.   
By Proposition \ref{blowup} we have that, for suitable rotations $R_k\in O(\R^{n-1})$ converging to the identity 
\begin{equation}\label{blowupconv2}
u_\e(p_k + \lambda_k R_k y) \to g_\e \left(\frac{\nu_L(P_\infty) \cdot y}{y_n}\right) \qquad
\text{ in } C^2_{loc}\big(\R^{n-1} \times (0,+\infty)\big),
\end{equation}
 as  $k\to\infty$. 
  By construction of $p_k$, we have that  
$P_k$ belongs to the plane generated by $e_n:=(0, \ldots, 1)$ and $\nu_L(p_k)$ and passing through $p_k$.
 Clearly $P_k = p_k + \lambda_k R_k y_k$ for some  $y_k\in \R^{n-1}\times (0,\infty)$ with $|y_k|=1$. Up to subsequence we have $y_k\to y_\infty$ for some $y_\infty\in \R^{n-1}\times [0,\infty)$ with $|y_\infty|=1$.
Now we assume that $u_\e(P_k) = 0$, i.e., $P_k\in\Sigma_\e$ for all $k$, and we wish to show that $y_\infty = e_n$. First, we claim that   $y_\infty \cdot e_n \neq 0$.  Indeed, if by contradiction,  $y_\infty \cdot e_n = 0$,
then we would have $y_k \cdot e_n  \to 0$, so that
$$
1 = \lim_k |\nu_L(p_k) \cdot R_k y_k| = \lim_k |\nu_L(P_\infty) \cdot y_k| = |\nu_L(P_\infty) \cdot y_\infty|.
$$
In particular, we would have $ \nu_L(P_\infty) \cdot y_\infty\neq 0$, \, $y_\infty \in \mathbb{R}^{n-1}\times \{ 0\}$ and $y_\infty=\pm \nu_L(P_\infty)$. 
Since  $P_k\in \Sigma_\e$, the last statement in Proposition \ref{blowup} would give a contradiction, and this proves the claim. Now  $y_\infty \cdot e_n \neq 0$, hence \eqref{blowupconv2}
yields
$$
g_\e\left(\frac{y_\infty \cdot \nu_L(P_\infty)}{y_\infty \cdot e_n}\right) = \lim_k u_\e (p_k + \lambda_k R_k y_k) = 0, 
$$
which gives $y_\infty \cdot \nu_L(P_\infty) = 0$. Since $y_\infty$ belongs to the vector space generated by $e_n$ and  $\nu_L(P_\infty)$ and it has unit length, we conclude that $y_\infty = e_n$, and the whole sequence $y_k$ converges to $e_n$ as $k\to\infty$.
\par
We are in the position to conclude the proof of the regularity of $\Sigma_\e$ near $\sinf$, and its orthogonality property. Indeed, since 
$$
\nabla g_\e\left(\frac{y \cdot \nu_L(P_\infty)}{y \cdot e_n}\right)  \neq 0 \qquad \text{ for } y= y_\infty = e_n,
$$
by \eqref{blowupconv2}, we deduce  that  also $\nabla u_\e(P_k)\neq 0$ for $k$ large enough. Thus, as the sequence $\{ P_k\}$ can be chosen arbitrarily, we conclude that $\Sigma_\e$ is smooth near the sphere at infinity by the implicit function theorem, with a well defined normal vector field $\nu_{\Sigma_\e}(P) = \frac{\nabla u_\e(P)}{|\nabla u_\e(P)|}.$ 
\par
Now, since $y_k \to e_n$, by \eqref{blowupconv2} we deduce,
\begin{equation}\label{orto}
\nu_{\Sigma_\e}(P_k) = \frac{\nabla u_\e(P_k)}{|\nabla u_\e(P_k)|} \to 
 \nu_L(P_\infty), 
\end{equation}
i.e., the normal vector field extends continuously up to the boundary, and this is enough to conclude that $\Sigma_\e \cup L$ is a $C^1$ hypersurface with boundary. 
\par
Finally, we prove the asymptotic expansion \eqref{abe}, using a blow-up argument analogous to that used to prove \eqref{blowupconv2}.   Let $\{P_k\}\in\Hyp^n$ converging  to some $P_\infty \in \sinf$.  
If $P_\infty \not\in L$ the proof is straightforward, since  for $P_\infty\in\Om^\pm$ we have that $u_\e(P_k) \to \pm 1$ and $\tilde{d}(P_k,K(L))\to \pm\infty$. 
\par
Now, we consider the case $P_\infty\in L$,
working as above in the half space model with origin in $P_\infty$, so that $K(L)$ is the cone over $L$ from the point $e_n$. Let $p_k\in L$  be points of minimal Euclidean distance from $P_k$,  let $\lambda_k=|p_k-P_k|$, and let $R_k\in O(\R^{n-1})$ such that $R_k \nu_L(P_\infty) = \nu_L (p_k)$ and $R_k \to Id$,  as in  \eqref{blowupconv2}.
 Again,  $P_k = p_k + \lambda_k R_k y_k$ for some  $y_k\in \R^{n-1}\times (0,\infty)$ with $|y_k|=1$ and  (up to a subsequence)  $y_k\to y_\infty$ for some $y_\infty\in \R^{n-1}\times [0,\infty)$ with $|y_\infty|=1$.
\par Now we distinguish two cases, corresponding to $y_\infty \cdot e_n = 0$ and $y_\infty \cdot e_n \neq 0$. 
If $y_\infty \cdot e_n = 0$, then, arguing as above, we have $y_\infty = \pm \nu_L(P_\infty)$, and hence  by the last statement in Proposition \ref{blowup} we have $u_\e(P_k) \to sgn ( \nu_L(P_\infty)\cdot y_\infty)$. Thus, we have    to prove that also $h_\e(\tilde d(P_k,K(L)))\to sgn ( \nu_L(P_\infty)\cdot y_\infty)$. To this purpose,  
it is enough to  notice  that $P_k$ lies always (for $k$ large enough) on the same side of $K(L)$ and that $d(P_k,K(L)) = d(R_k y_k, \lambda_k^{-1} (K(L) - p_k) ) \to \infty $  as $k\to \infty$, since $R_k y_k\to y_\infty$ while 
$ \lambda_k^{-1} (K(L) - p_k)$ approaches the vertical half plane passing through the origin and orthogonal to $y_\infty$.
\par
We pass to consider the case $y_\infty \cdot e_n \neq 0$. Since 
$$
g_\e \left(\frac{ \nu_L(p_k) \cdot R_k y_k}{R_k y_k\cdot e_n}\right) 
\to
g_\e \left(\frac{\nu_L(P_\infty) \cdot y_\infty}{y_\infty \cdot e_n}\right)
\qquad \text{ as } k\to\infty,
$$
thanks to the blow-up formula given by 
\eqref{blowupconv2} for $y=y_\infty$, 
it is clearly enough to prove that
\begin{equation}\label{hauconv1}
g_\e \left(\frac{ \nu_L(p_k) \cdot R_k y_k}{R_k y_k\cdot e_n}\right) 
-
h_\e\left( \tilde d\big(R_k y_k, \lambda_k^{-1} (K(L) - p_k) \big) \right) \to 0
\qquad \text{ as } k\to\infty.
\end{equation}
Let us set $K^\infty(L)$ the cone over $K$ from the point at infinity, i.e., $K^\infty(L) = L\times (0\times\infty)$. Then, it is easily seen that by construction of $p_k$, we have $\tilde d(P_k,K^\infty(L) ) = \tilde d(P_k, T_{p_k}L\times (0,\infty))$. Therefore, 
$$
g_\e \left(\frac{ \nu_L(p_k) \cdot R_k y_k}{R_k y_k\cdot e_n}\right) = 
h_\e\left(\tilde d(P_k, T_{p_k}L\times (0,\infty))\right) =
h_\e\left(\tilde d(P_k,K^\infty(L))\right).
$$
Since $\tilde d(P_k,K^\infty(L))=\tilde d\big(R_k y_k, \lambda_k^{-1} (K^\infty(L) - p_k) \big)$,  \eqref{hauconv1} is equivalent to 
\end{proof}
\begin{equation}\label{hauconv2}
h_\e\left( \tilde d\big(R_k y_k, \lambda_k^{-1} (K^\infty(L) - p_k) \big) \right) 
-
h_\e\left( \tilde d\big(R_k y_k, \lambda_k^{-1} (K(L) - p_k) \big) \right) \to 0
\end{equation}
as $k\to\infty$.  In order to prove  \eqref{hauconv2}, it is enough to check the Hausdorff convergence on compact sets  (usually referred to as Kuratowsky convergence) in $\R^{n-1}\times (0,\infty)$ of  $\lambda_k^{-1} (K^\infty(L) - p_k) $ and $\lambda_k^{-1} (K(L) - p_k)$ to $T_{P_\infty}L\times (0,\infty)$. Finally, this Hausdorff convergence  is indeed a direct consequence of the fact that  $K^\infty(L)$ and $K(L)$ are tangent along $L$, since they both touch the sphere at infinity orthogonally along the smooth hypersurface $L$; for sake of brevity  we  skip the details which are standard.

\section{Minimal hypersurfaces}
In this final section we study the limit when $\e$ tends to zero. First we investigate the behaviour of the energy functionals $\E_\e$ using $\Gamma$-convergence and we prove Theorem \ref{gcthm}. Then we apply this result to the local minimizers $u_\e$ to construct entire minimal hypersurfaces $\Sigma_\e$ with prescribed boundary at infinity and we prove Theorem \ref{micorre}.

\subsection{Proof of the $\Gamma$-convergence result}
Here we prove the $\Gamma$-convergence result given by Theorem \ref{gcthm}. The proof relies on the very well known arguments in the Euclidean setting \cite{M}, with some care in order to treat the boundary conditions $v_\e=w_\e$ on $\partial B_R$.  We divide the proof in several steps, using the same notations defined in the Introduction.
\vskip3pt
\noindent
{\it Step 1 (Compactness.)}
Since the metric on compact subsets of $\Hyp^n$ is equivalent to the Euclidean one, clearly  we may assume that 
$$
\epsm \int_{B_R} \frac12 |\nabla v_\epsm |^2+W_\epsm(v_\epsm) dx \leq C \, ,
$$
where $C>0$ depends only on $R$. Since $ |v_\epsm|\le 1$, then arguing as in \cite{M}, Proposition 3, up to subsequence we have $v_\epsm \to v^*$ in $L^1(B_R)$, where $v^* \in BV(B_R;\{-1,+1\})$. 
\vskip3pt
\noindent 
{\it Step 2 ($\Gamma$-liminf.)}
By {\it Step 1} we may  assume $v\in BV(B_R;\{-1,+1\})$, i.e.,  $\mathcal{F}(v;w^*, B_R)<\infty$. Moreover,  we may assume that each $v_\epsm$ has finite energy in $B_R$ and $|v_\epsm| \leq 1$ a.e., because energy decreases under truncation and truncation keeps the boundary conditions $v_\epsm=w_\epsm$ on $\partial B_R$. Now we essentially follows \cite{M} but with some extra care because of the possible jump between $v$ and $w^*$ along $\partial B_R$.
\par Let $A \subset \subset B_1$ be an open set with compact closure such that  $\overline{B_R} \subset A$ and let $\Psi(t) =\int_0^t \sqrt {W(s)} ds$,  so that $\Psi \in C^1 (\mathbb{R})$ and it is an odd function. We consider $\hat v_\epsm \in H^1(A)$ as $v_\epsm$ extended as $w_\epsm$ outside $B_R$. Since $\hat v_\epsm \in H^1(A)$, by the chain rule in $H^1(A)$ the functions $\Psi(\hat v_\epsm )$ satisfy $\Psi( \hat v_\epsm) \in W^{1,1}(A) \subset BV(A)$ and
$$
2| \Psi(\hat v_\epsm) |_{BV_g(A)}= 2\int_{A}  \sqrt{W(\hat v_\epsm)} \| \nabla_g \hat v_\epsm \| dVol_g \leq \sqrt{2}\epsm \E(\hat v_\epsm,A)=\mathcal{F}_\epsm(\hat v_\epsm; w_\epsm,A)  \, .
$$
Taking \eqref{datalbo} into account we have $\hat v_\epsm \to \tilde v={\tilde v}_{w^*}$  in $L^1(A)$ as $m \to \infty$.   By lower semicontinuity of the total variation and using the pointwise equality $2 \Psi(\tilde v)=C_W \tilde v $  we obtain
\begin{equation}
\label{quasigliminf}
C_W | \tilde v|_{BV_g(\overline{B_R})} \leq C_W | \tilde v|_{BV_g(A)} = 2 | \Psi(\tilde v) |_{BV_g(A)} \leq \liminf_\epsm  \mathcal{F}_\epsm(v_\epsm; w_\epsm,A).
\end{equation}

\par Finally, since $\mathcal{F}_\epsm(v_\epsm; w_\epsm,A)=\mathcal{F}_\epsm(v_\epsm; w_\epsm,B_R)+\mu_\epsm (A\setminus B_R))$ and $\mu^*(\partial B_R)=0$ the conclusion follows from \eqref{quasigliminf} and \eqref{datalbo} when $A=B_{\rho}$ and $\rho \searrow R$.

\vskip3pt
\noindent
{\it Step 3 ($\Gamma$-limsup without boundary conditions.)} 
Here we show that, for any given function $v\in BV(B_R;\{-1,+1\})$,  there exists a sequence $v_\epsm \in H^{1}(B_R)$ with $|v_\epsm|\le 1$, $v_\epsm \to v$ in $L^1(B_R)$ and such that 
\begin{equation}\label{glsfree}
\sqrt{2} \epsm \E_\epsm(v_\epsm, B_R) \to C_W |v|_{BV_g(B_R)} \qquad \text{ as } \epsm\to 0. 
\end{equation}
This $\Gamma$-limsup inequality is well understood in the Euclidean setting. The proof in the present case could be obtained by standard localization arguments, freezing the $x$ dependence in the energy density functionals. Here for the reader convenience we sketch the original  proof in \cite{M}, adapting it to the hyperbolic setting. By standard density arguments in the Euclidean setting and formulas \eqref{hyptotvar} and \eqref{halfsptotvar},  the class of functions in $BV_g(B_R;\{-1,+1\})$ with smooth jump set are actually dense in $L^1$ and in energy.  
Therefore, by diagonal arguments in $\Gamma$-convergence we can prove \eqref{glsfree}
assuming $S_v$ smooth.  In this case, following \cite[Proposition 2]{M} it turns out that a recovery sequence is given by 
$v_\epsm(x)=h_\epsm(\tilde d(x, S_v))$, where $h_\epsm$ is the optimal one-dimensional profile given by Proposition \ref{odsprop}, and $\tilde d(x, S_v) =  v(x) d(x, S_v)$ is the hyperbolic  signed distance from $S_v$ (unique up to the sign).   
\vskip3pt
\noindent
{\it Step 4 ($\Gamma$-limsup with boundary conditions).}
In this step we construct a recovery sequence taking into account the boundary conditions. To this purpose let $v\in BV(B_R;\{-1,+1\})$. First we show that the class of functions coinciding with $w^*$ in a neighborhood of $\partial B_R$ are dense in energy and in $L^1(B_R)$.  Indeed,  let $0<\lambda<1$ and set
$$
v_\lambda(x):= 
\begin{cases}
v(x) & \text{ if } |x|\le \lambda R;\\
w^*(x) & \text{ otherwise.}
\end{cases}
$$
Then we have  $v_\lambda\equiv w^*$ near $\partial B_R$ and  $v_\lambda \to v$ in $L^1(B_R)$ as $\lambda \nearrow 1$. Moreover, since  $\mu^*(\partial B_R)= 0$ 
it is easy to prove that $|D_g v_\lambda|(\partial B_{\lambda R}) \to |D_g \tilde v_{w^*}|(\partial B_{R})$, so that 
$\F(v_\lambda ; w^*, B_R) \to \F(v ; w^*, B_R)$ as $\lambda \nearrow 1$.
Therefore, up to a further diagonal argument,  
  without loss of generality we may assume $v\equiv w^*$ in a neighborhood of $\partial B_R$, so that $\F(v;w^* ,B_R) = C_W |v|_{BV_g(B_R)}$. 
\par
Now we aim to  glue together the recovery sequence $v_\epsm \to v$ constructed  in {\it Step 3} with $w_\epsm$, in order to obtain a recovery sequence  which takes into account the boundary conditions. To this purpose, for any fixed $\eta>0$ we construct an approximated recovery sequence $\hat v_\epsm \in H^1_{w_\epsm}(B_R)$  (depending on $\eta$), with $ \hat v_\epsm \to v$ in $L^1(B_R)$ as $m\to\infty$ and   satisfying
\begin{equation}\label{gammadia}
\limsup_m \F_\epsm(\hat v_\epsm; w_\epsm, B_R)  \le   \F( v; w^*, B_{R}) + C\eta.
\end{equation}
Then, the $\Gamma$-limsup inequality  follows from \eqref{gammadia}  by a standard diagonal argument as $\eta \to 0$. 
\par
To prove \eqref{gammadia} let $\delta=\delta(\eta)>0$ be so small  such that the following holds.
\begin{itemize}
\item[i)]
  $v=w^*$ in $C_\delta:= B_R\setminus \overline{B_{R-\delta}}$;
\item[ii)]  
 $\E_\epsm(w_\epsm, C_\delta)\le \eta \epsm^{-1}$ for every $m$; 
\item[iii)]    $\E_\epsm(v_\epsm, C_\delta)\le \eta \epsm^{-1}$ for every $m$.
\end{itemize}  
Notice that, for $\delta$ suitably small  ii) holds since $\mu^*(\partial B_R)=0$,  and iii) is  true since  $v_\epsm$ is a recovery sequence  for $v$ in $B_R$, and therefore also in $C_\delta$, and $|v|_{BV(C_\delta)}\to 0$ as $\delta \to 0$.
For  each $m$ we divide the annulus $C_\delta$ in $M_m:=[\frac{\delta}{\eta \epsm}]$ (where $[\cdot]$ is the integer part) concentric annuli of thickness $\tilde\e_m := \frac{\delta}{M_m}$. In this way we clearly have $\tilde \e_m = \eta_m \epsm$ with $\eta_m\to \eta$ as $m\to \infty$. Since $v_\epsm\to v=w^*$ and $w_\epsm \to w^*$ in $L^1(C_\delta)$ we have $v_\epsm - w_\epsm \to 0$ in $L^1(C_\delta)$. Therefore, by  the mean value theorem  we can choose $k_m\in \{1, \ldots M_m\}$ such that  
\begin{equation}\label{stimal1}
 \frac{1}{{\tilde\e}_m} \int_{\tilde C_{k_m}} |v_\epsm - w_\epsm| \, dx \to 0 \qquad \text{ as } m\to \infty,
\end{equation}
where  $\tilde C_{k_m} = \{R'_m < |x| < R''_m\}$ and $R'_m:= R - \delta + (k_m -1)  {\tilde\e}_m$, $R''_m := R- \delta + k_m  {\tilde\e}_m$.   
Let $\f_\epsm$ be a radial Lipschitz cut-off function such that $\f_{\e_m}(x)\equiv 0$ for $|x|\ge 
R''_m$, $\f_{\e_m}(x)\equiv 1$ for $|x|\le 
R'_m$, and decreases linearly along the rays in $\tilde C_{k_m}$. For all $x\in B_R$ we  set
$$
\hat{v}_\epsm(x):= v_\epsm(x) \f_\epsm(x) + w_\epsm(x) (1 - \f_\epsm(x)).
$$
By construction we have for a.e.  $x\in B_R$
$$
| \nabla \hat{v}_\epsm(x) | \le |\nabla v_\epsm| + |\nabla w_\epsm| + \frac{1}{{\tilde\e}_m} |v_\epsm - w_\epsm|.
$$
Then, by ii) and iii) above, by Young inequality, and the bounds $|w_\epsm|\le 1$, $|v_\epsm|\le 1$ in $B_R$ we obtain
\begin{multline*}
\E_\epsm(\hat v_\epsm, \tilde C_{k_m}) \le C \int_{\tilde C_{k_m}} |\nabla \hat v_\epsm|^2 + W_\epsm(\hat v_\epsm) \, dx \le 
C \int_{C_\delta} |\nabla v_\epsm|^2 + |\nabla w_\epsm|^2 \, dx + \\
+C  \int_{\tilde C_{k_m}} \frac{1}{({\tilde\e}_m)^2}|v_\epsm - w_\epsm|^2 + \frac{1}{(\epsm)^2} \, dx\le C\eta \epsm^{-1}
+ \frac{C \epsm^{-1}}{\eta}   \int_{\tilde C_{k_m}} \frac{1}{{\tilde\e}_m}|v_\epsm - w_\epsm| \, dx . 
\end{multline*}
Therefore, by \eqref{stimal1}, for $m$ large enough (depending only on $\eta$) we have
\begin{equation}\label{stietafix}
\epsm \E_\epsm(\hat v_\epsm, \tilde C_{k_m}) \le C\eta
\end{equation}
By \eqref{stietafix}, in view of ii) above  we have 
\begin{multline*}
\epsm \E_\epsm(\hat v_\epsm, B_R) = \epsm  \E_\epsm( v_\epsm, B_{R'_m}) + \epsm\E_\epsm(\hat v_\epsm, \tilde C_{k_m}) + \\
+ \epsm\E_\epsm(w_\epsm, B_R\setminus \overline{B_{R''_m}}) \le
\epsm  \E_\epsm( v_\epsm, B_{R}) + C\eta
\end{multline*}
Passing to the limit for $m\to\infty$, we obtain
$$
\limsup_m \F_\epsm(\hat v_\epsm; w_\epsm, B_R)  \le  \limsup_m \sqrt{2} \epsm  \E_\epsm( v_\epsm, B_{R}) + C\eta = \F( v; w^*, B_{R}) + C\eta,
$$
so that \eqref{gammadia} holds, and this concludes the proof of the $\Gamma$-limsup inequality. 


\begin{oss}\label{assmuzero}
The assumption $\mu^*(\partial B_R)=0$ is essential in order to identify the boundary  term in the $\Gamma$-limit $\F$. Indeed, for $w^*$ equal to $1$ in $B_R$ and  $w^* = -1$ in 
$B_1\setminus B_R$, it is very easy to construct two approximating sequences $w^{\pm}_\epsm$ for $w^*$ satisfying \eqref{datalbo},  with traces on $\partial B_R$ equal to $\pm1$, respectively. Therefore,  the corresponding $\Gamma$-limit    is clearly given by  \eqref{gali} with $w^*$ replaced by $\pm 1$ on $B_1\setminus B_R$, respectively.  
More  generally, given $\tilde w$ and  $w^*$  it is always possible to  construct an   approximating sequences $w_\epsm$ for $w^*$ such that  the corresponding $\Gamma$-limit is given by  \eqref{gali} with $w^*$ replaced by $\tilde w$.  Thus we see that, removing  the assumption $\mu^*(\partial B_R)=0$,  the $\Gamma$-limit  may depend on the whole sequence $w_{\e_m}$ and not only on $w^*$.
\end{oss}

\subsection{Existence and asymptotic behavior of minimal hypersurfaces}
In this final part, we prove the existence of an entire minimal hypersurface with prescribed behaviour at infinity. First we give a local  energy bound for the minimizers $u_\e$ which allows to obtain a limiting function $u^*\in BV_{loc}(B_1;\{-1,1\})$ with the desired behaviour at infinity. Then, we can apply the $\Gamma$-convergence result in the previous subsection to get the area-minimizing property of the jump set $S_{u^*}$ and to conclude the proof of Theorem \ref{micorre}.

\begin{lemma}\label{unicobo}
Let  $0<R<1$, let $\e < R/2$ and let $u_\e$ be a  local minimizer of \eqref{enfun}. Then we have
$ \E_{\e}(u_\e, B_R) \le {C}\e^{-1}$, where $C$ is a constant depending only on $R$.
\end{lemma}
\begin{proof}
Since $u_\e$ are uniformly bounded, by equation \eqref{equa3} and by standard elliptic regularity we have that $|\nabla u_\e|\le c \e^{-1}$, where $c$ depends only on $R$.  We deduce that  
$$
 \| \nabla u_\e \|^2 + W_{\e} (u_\e) \le \frac{c}{\e^2} \qquad \text{ in } B_R,
$$
where $c$ depends only on $R$. Let $\f_\e$ be a radial cut-off function, equal to $1$ for $|x|\le R-\e$, and decreasing linearly to zero along rays for $R-\e \le |x| \le R$. Let us set $v_\e:= \f + (1-\f) u_\e$. By construction we have that $v_\e= u_\e$ on $\partial B_R$, and $|\nabla v_\e|\le c \e^{-1}$, with $c$ depending only on $R$.   Thus, by local energy minimality of $u_\e$ we have 
$$
\E_\e(u_\e, B_R)\le \E_\e (v_\e, B_R) \le \E_\e(v_\e,B_R\setminus B_{R-\e}) \le \frac{C}{\e},
$$ 
where $C$ is a constant depending only on $R$. 
\end{proof}

\begin{proof}[Proof of Theorem \ref{micorre}] 
We will prove claims i) and ii) of the theorem separately, using the Poincar\'e ball model. Claim iii) is well known and it has been already discussed in the Introduction.
\vskip3pt
\noindent
{\it Proof of i).} 
By Lemma \ref{unicobo} we have $\F_\epsm(u_{\epsm};u_{\epsm}, B_R)\le C$, and hence by Theorem \ref{gcthm}, i),  passing to a subsequence we have   $u_\epsm\to u^*$ in $L^1(B_R)$ for some $u^*\in BV(B_R;\{-1,1\})$. Thus a simple diagonal argument yields
$u_\epsm\to u^*$ in $L^1_{loc}(B_1)$ for some $u^*\in BV_{loc}(B_1;\{-1,1\})$ as $m \to \infty$. 
\par 
We pass to the proof of  $ S_{u^*}\subset \overline{\text{conv}(F)}$. To this purpose, notice that if $ x\in B_1\setminus  \overline{\text{conv}(F)}$, then there is  a neighborhood $N_x$ of $x$ compactly contained in $B_1\setminus  \overline{\text{conv}(F)}$, and
there exists an elementary solution $U_\e$ 
such that either $0< U_\e \le u_\e\le 1$ with  $U_\e \to 1$ uniformly in $N_x$, or
$-1 \le u_\e \le U_\e<0$ with  $U_\e \to -1$ uniformly in $N_x$. 
In both cases we deduce that $u^*$ is constant in $N_x$,  so that in particular 
 $S_{u^*} \cap N_x = \emptyset$. By the arbitrariness of $ x\in B_1\setminus  \overline{\text{conv}(F)}$ we conclude  $ S_{u^*}\subset \overline{\text{conv}(F)}$, which clearly implies $ \overline{ S_{u^*}} \cap \sinf \subset F$ . 
\par
It remains only to prove the inclusion $ \partial \Om^+ \cap \partial \Om^- \subseteq \overline{ S_{u^*}}$. Let $p\in 
\partial \Om^+ \cap \partial \Om^-$, and for any given $\delta > 0 $ let us fix two points  $q^\pm \in \Om^\pm$  such that $q^\pm \in I_\delta(p)$. Moreover, let $0<\rho<\delta$ be such that $B_{\rho}(q^+) \cap B_{\rho}(q^-) = \emptyset$ and $B_{\rho}(q^\pm) \cap B_1 \cap  \overline{\text{conv}(F) }= \emptyset$.  Set  $B^\pm:= B_{\rho}(q^\pm) \cap \sinf\subset \Om^\pm$ and  consider the tube $T_{\rho} := \overline{\text{conv} (B^+ \cup B^-)} $. Then, since $u^* = \pm 1$ on $B_{\rho}(q^\pm) \cap T_\rho$, we have  that  $u^*$ takes both values $+1$ and $-1$ on sets of positive measure in  $T_{\rho}$, that clearly implies 
 $|D u^*|(T_{\rho}) = \hs^{n-1}(S_{u^*}\cap T_{\rho}) >0$. Therefore $B_{C \delta} (p)\cap S_{u^*} \neq \emptyset $ where $C$ is a constant independent of $\delta$, and this,  by the arbitrariness of $\delta$, yields $p\in \overline{S_{u^*}}$, which  concludes the proof of property {\it i)}.
\vskip3pt
\noindent
{\it Proof of  ii).}
Let $u^*$ as given by part $i)$ and let $v^*\in BV_{\rm{loc}}(B_1;\{ -1, 1\})$ such that the support of $u^*-v^*$ is compactly contained in $B_R$ for some $R \in (0,1)$. 
Note that, since $u_\epsm \to u^*$ in $L^1_{loc}(B_1)$ and Lemma \ref{unicobo} holds, passing to a subsequence if necessary, we may assume that $\{u_\epsm\}$ satisfies assumption \eqref{datalbo} (with $w_\epsm=u_\epsm$).
Thus, changing $R$ slightly if necessary, we may also assume that $\mu^*(\partial B_R)=0$, because such condition may fail for at most countably many radii. Let $v_\epsm \to v^*$ in $L^1(B_R)$ be a recovery sequence for $v^*$ in $B_R$ with $u_\epsm$ as boundary data as given by  Theorem \ref{gcthm}. Thus, combining the $\Gamma$-convergence result given by  Theorem \ref{gcthm}  and the energy minimality of each $u_\epsm$ we obtain
$$
 \F(u^*; u^*,B_R)\le \liminf_\epsm \F_\epsm(u_\epsm;u_\epsm,B_R)\le  \limsup_\epsm \F_\epsm(v_\epsm;u_\epsm,B_R) \leq \F(v;u^*,B_R)  \, ,
$$
i.e. $|u^*|_{BV_g(B_R)} \leq |v^*|_{BV_g(B_R)} $ as claimed.
\end{proof}

\section*{Acknowledgments}
We wish to thank Lorenzo Bertini, Gabriele Mondello e Andrea Sambusetti for interesting discussions on the subject.

\end{document}